\documentclass[reqno,11pt]{article}
\usepackage{a4wide,color,eucal,enumerate,mathrsfs}
\usepackage[normalem]{ulem}
\usepackage{amsmath,amssymb,epsfig,bbm}
\usepackage{pdfsync}
\numberwithin{equation}{section}
\usepackage[pdfborder={0 0 0}]{hyperref}

\newcommand{\C}{\mathbb{C}}

\newcommand{\R}{\mathbb{R}}

\newcommand{\mm}{{\mbox{\boldmath$m$}}}

\newcommand{\ppi}{{\mbox{\boldmath$\pi$}}}

\newcommand{\restr}[1]{\lower3pt\hbox{$|_{#1}$}}

\newcommand{\Leb}[1]{{\mathscr L}^{#1}}      

\newcommand{\eps}{\varepsilon}  
\newcommand{\nchi}{{\raise.3ex\hbox{$\chi$}}}

\newcommand{\media}{\mkern12mu\hbox{\vrule height4pt           %
          depth-3.2pt                                 
          width5pt}\mkern-16.5mu\int\nolimits}        
\newcommand{\mediatext}[1]{- \hskip-.9em \int_{#1}} 

\newcommand{\Probabilities}[1]{\mathscr P(#1)}          

\newenvironment{proof}{\removelastskip\par\medskip   
\noindent{\em Proof.}
\rm}{\penalty-20\null\hfill$\square$\par\medbreak}

\newtheorem{theorem}{Theorem}[section]

\newtheorem{lemma}[theorem]{Lemma}
\newtheorem{proposition}[theorem]{Proposition}
\newtheorem{definition}[theorem]{Definition}

\newtheorem{remark}[theorem]{Remark}

\newcommand{\fr}{\hfill$\blacksquare$}                      

\renewcommand{\C}{{\rm Ch}}

\renewcommand{\C}{{\sf Ch}}

\renewcommand{\mm}{\mathfrak m}

\title{Tensorization of Cheeger energies, \\the space $H^{1,1}$ and the area formula for graphs}
\begin{document}

\author{Luigi Ambrosio\
   \thanks{Scuola Normale Superiore, Pisa. email: \textsf{luigi.ambrosio@sns.it}}
   \and
   Andrea Pinamonti\
 \thanks{Scuola Normale Superiore, Pisa. email: \textsf{andrea.pinamonti@sns.it}}
 \and
   Gareth Speight\
 \thanks{Scuola Normale Superiore, Pisa. email: \textsf{gareth.speight@sns.it}}
   }

\maketitle
\begin{abstract} First we study in detail the tensorization properties of weak gradients in metric measure spaces $(X,d,\mm)$.
Then, we compare potentially different notions of Sobolev space $H^{1,1}(X,d,\mm)$ and of weak gradient with exponent 1. 
Eventually we apply these results to compare the area functional $\int\sqrt{1+|\nabla f|_w^2}\,d\mm$ with the perimeter of the subgraph
of $f$, in the same spirit as the classical theory. 
\end{abstract}


\section{Introduction}

The aim of this paper is the investigation of some fine analytic questions related to the theory of weak gradients in
metric measure spaces. One of our motivations has been the study of the area functional
$$
{\cal A}(f)=\int_X\sqrt{1+|\nabla f|_w^2}\,d\mm
$$
in metric measure spaces $(X,d,\mm)$, where $|\nabla f|_w$ denotes a suitable notion of weak gradient
induced by the metric measure structure. In the metric measure setting, the functional ${\cal A}$ 
has been an object of investigation in the recent paper \cite{HKL13}, see also \cite{HKLL14} for generalizations. 
In the classical theory, it is well known that ${\cal A}$ corresponds to
the surface area of the subgraph 
$$
E_f:=\bigl\{(x,t)\in X\times\R:\ t<f(x)\bigr\}
$$
of $f$. One of the aims of this paper is to investigate to what extent this correspondence
holds also in metric measure spaces. The question makes sense, because by now there is a well established theory of
$BV$ functions and sets of finite perimeter in metric measure spaces, initiated in \cite{Mi} (see also \cite{AmDi}, \cite{KLS13}
for more recent developments). This theory provides, among other things, a canonical definition of surface area for
sets of finite perimeter, see Section~\ref{ssec4} for details.

Heuristically, the representation of ${\cal A}(f)$ as surface area of a subgraph seems to be closely related to a tensorization
property of the weak gradients in the product metric measure structure, where one of the factors is 
simply the Euclidean real line endowed with the Lebesgue
measure $\Leb{1}$. Specifically, we mean the following: if $(X,d_X,\mm_X)$ and $(Y,d_Y,\mm_Y)$ are the factors, then
for the metric measure structure
$$
\bigl(X\times Y,\sqrt{d_X^2+d_Y^2},\mm_X\times\mm_Y\bigr)
$$
the weak gradient (for the sake of simplicity we ignore in this introduction its potential dependence on the integrability exponent)
should satisfy
\begin{equation}\label{eq:alicaju}
|\nabla f(x,y)|^2_w=|\nabla f(\cdot,y)|_w^2(x)+|\nabla f(x,\cdot)|_w^2(y)\qquad\text{for $\mm_X\times\mm_Y$-a.e. $(x,y)\in X\times Y$.}
\end{equation}
There exist many properties that are easily seen to be stable under tensorization: compactness, properness, completeness, separability, length and geodesic properties,
doubling, Poincar\'e inequalities, etc. As a matter of fact, although there are heuristic arguments suggesting that tensorization should always be
true for the weak gradients (see Proposition~\ref{prop:easy_tensorization}, dealing with Euclidean spaces endowed with general norms), this question
has not been much investigated so far. In \cite{AGS11b} the tensorization property has been proved assuming a curvature lower
bound on the factors $(X,d_X,\mm_X)$, $(Y,d_Y,\mm_Y)$ and the quadratic structure of the Cheeger energy. 
In Section~\ref{sec:tenso} we refine the analysis of \cite{AGS11b} 
and we prove three results, all independent of curvature assumptions (with the first and third also independent of doubling
and Poincar\'e assumptions, the first and the second independent of quadraticity assumptions):

\noindent
(a) The weak gradient $|\nabla f|_w$ in the product structure always coincides with the $L^2(X\times Y,\mm_X\times\mm_Y)$
relaxation of the squared ``cartesian'' slope
$$
|\nabla f|^2_c(x,y):=|\nabla f(\cdot,y)|^2(x)+|\nabla f(x,\cdot)|^2(y),
$$
starting 
from the class of locally Lipschitz functions (here $|\nabla g|$ is the slope, or local Lipschitz constant, defined in 
\eqref{def:slope}), see Theorem~\ref{thm:I=C}.

\noindent
(b) Using (a), we show that if the factors are doubling and satisfy a $(1,2)$-Poincar\'e inequality, then \eqref{eq:alicaju} holds, see Theorem~\ref{thm:PI_tensor}.
Although we follow a different path for the proof, it might be that this result could also be obtained starting from Proposition~\ref{prop:easy_tensorization}
using Cheeger's differentiable structure and a suitable notion of product of Cheeger's charts. We also prove a more refined result, namely if $f\in L^2(X\times Y,\mm_X\times\mm_Y)$
satisfies $f(\cdot,y)\in H^{1,2}(X,d_X,\mm_X)$ for $\mm_Y$-a.e. $y\in Y$, $f(x,\cdot)\in H^{1,2}(Y,d_Y,\mm_Y)$ for $\mm_X$-a.e. $x\in X$
and
$$
\int_X\int_Y |\nabla f(x,\cdot)|_w^2\,d\mm_Y d\mm_Y(x)
+
\int_Y\int_X |\nabla f(\cdot,y)|_w^2\,d\mm_X d\mm_Y(y)<\infty,
$$
then $f\in H^{1,2}(X\times Y,\sqrt{d_X^2+d_Y^2},\mm_X\times\mm_Y)$.

\noindent
(c) If the factors are strongly asymptotically Hilbertian, meaning that $\int_X |\nabla f|^2\,d\mm_X$ and 
$\int_Y|\nabla g|^2\,d\mm_Y$ are quadratic forms on locally Lipschitz functions, then the tensorization property \eqref{eq:alicaju}
holds, see Theorem~\ref{thm:Cquad}. 

Section~\ref{sec:4} is devoted to the analysis of the Sobolev space $H^{1,1}(X,d,\mm)$. We define $H^{1,1}(X,d,\mm)$
in the spirit of the $BV$ theory, considering those functions whose total variation $|Df|$ is absolutely continuous w.r.t.
$\mm$, and which are $H^{1,1}$ along almost every curve (in the sense of \cite{AGS11a}). Then, in Theorem~\ref{thm:equivSob1} we compare
this definition with the other ones already considered in the literature. An interesting fact is that, even though under doubling
and $(1,1)$-Poincar\'e the spaces are the same, the associated notions of weak gradient do not coincide (even though they are
comparable, see \eqref{eq:treponzio3} and \eqref{eq:treponzio4}), see \cite{HKLL14} and Remark~\ref{rem:do_not_coincide}.
More precisely, the gradient $|\nabla f|_{*,1}$ singled out by the $BV$ theory can be approximated
only in the sense of measures by slopes of Lipschitz functions (see \eqref{optimal_sequence_BV}) and, by construction, it gives rise to a functional
$f\mapsto \int_X|\nabla f|_{*,1}\,d\mm$ which is lower semicontinuous on locally Lipschitz functions. On the other hand,
the gradient $|\nabla f|_{w,1}$ singled out by the theory of functions absolutely continuous along ${\rm Mod}_1$-a.e. curve
enjoys a stronger approximation property (see Proposition~\ref{prop:strongapp} for the simple proof) but, in general, it fails to give raise to a lower semicontinuous functional $f\mapsto\int |\nabla f|_{w,1}\,d\mm$.

Finally, in Section~\ref{sec:per} we prove upper and lower bounds for the perimeter of $E_f$, even for $f\in BV(X,d,\mm)$. For
$f\in H^{1,1}(X,d,\mm)$, under the doubling and $(1,1)$-Poincar\'e assumptions we get
$$
\int_B\sqrt{1+|\nabla f|^2_{*,1}}\,d\mm\leq P(E_f,B\times\R)\leq
\int_B\sqrt{1+|\nabla f|^2_{w,1}}\,d\mm
$$
for any Borel set $B\subset X$,
so that equalities hold if and only if the two notions of gradient coincide for the metric measure structure. Besides the case of asymptotically
Hilbertian metric measure spaces
with a curvature lower bound (see \cite{GiHa}, which are the so-called $RCD(K,\infty)$ spaces introduced in \cite{AGS11b}),
we are not aware of conditions ensuring the coincidence $\mm$-a.e. of the
two gradients. 

\smallskip\noindent
{\bf Acknowledgements.} 
We would like to thank N. Shanmugalingam for helpful hints on the theory of weak gradients in the critical case $p=1$.
The authors acknowledge the support of the grant ERC ADG GeMeThNES. The second author has been partially supported by the Gruppo
Nazionale per l'Analisi Matematica, la Probabilit\`a e le loro
Applicazioni (GNAMPA) of the Istituto Nazionale di Alta Matematica
(INdAM).

\section{Notation and preliminary results}\label{sec:2}

We assume that all metric measure spaces $(X,d,\mm)$ are complete and separable, that $\mm$ is a finite Borel measure
and that the support of $\mm$ is the whole space $X$. 

We denote by $\Leb{d}$ the Lebesgue measure in $\R^d$. We say that $f:X\to\R$ is
{\it locally Lipschitz} if all its restrictions to bounded sets are Lipschitz. The {\it slope} (or
local Lipschitz constant) of $f:X\to\R$ is defined by
\begin{equation}\label{def:slope}
|\nabla f|(x)=\limsup_{y\to x}\frac{|f(y)-f(x)|}{d(y,x)}.
\end{equation}

\subsection{Doubling and Poincar\'e inequality}\label{ssec1}

We say that $(X,d,\mm)$ is doubling if there exists a constant $C>0$ satisfying
\begin{equation}\label{eq:doubling}
\mm(B_{2r}(x))\leq C\mm(B_r(x))\qquad\forall x\in X,\,\,r>0.
\end{equation}

Given $p\in [1,\infty)$, we say that the $(1,p)$-Poincar\`e inequality holds for Lipschitz functions if, for all
open balls $B_r(x)\subset X$, one has
\begin{equation}\label{eq:PI}
\media_{B_r(x)}| f-f_{x,r}|\,d\mm\leq c r\biggl(\media_{B_{\lambda r}(x)}|\nabla f|^p\,d\mm\biggr)^{1/p}
\qquad\forall f\in {\rm Lip}_{\rm loc}(X)
\end{equation}
for suitable $c\geq 0$ and $\lambda\geq 1$ independent of $B_r(x)$ (here $\mediatext{A}$ denotes the averaged integral on
a Borel set $A$ and $f_{x,r}=\mediatext{B_r(x)}f$).
In the sequel we say that a constant is structural if it depends only on the doubling constant in \eqref{eq:doubling}
 and the constants $c,\,\lambda$ and the exponent $p$ in \eqref{eq:PI}.

\subsection{Maximal functions and approximate continuity}\label{ssec2}

Given $g\in L^1(X,\mm)$ nonnegative, we define
\begin{equation}\label{def:massimale}
M_r g(x):=\sup\left\{ \media_{B_s(x)}g\,d\mm:\ s\in (0,r)\right\}\qquad x\in X,\,\,r\in (0,\infty].
\end{equation}
More generally, if $\mu$ is a finite Borel measure, we define
\begin{equation}\label{def:massimalemu}
M_r \mu(x):=\sup\left\{\frac{\mu(B_s(x))}{\mm(B_s(x))}:\ s\in (0,r)\right\}\qquad x\in X,\,\,r\in (0,\infty],
\end{equation}
so that $M_r(g\mm)=M_r g$.

For doubling metric measure spaces $(X,d,\mm)$ 
we will need the weak $L^1$ estimate for the maximal function; we need
it in the stronger form, for $M_\infty$: 
\begin{equation}\label{eq:maximal_estimates1}
\mm\bigl(\{M_\infty\mu>\lambda\}\bigr)\leq c\frac {\mu(X)}\lambda\qquad\forall\lambda>0.
\end{equation}
The constant $c$ in \eqref{eq:maximal_estimates1} depends only on the doubling constant of $(X,d,\mm)$.
For absolutely continuous measures we will also use the more refined estimate
\begin{equation}\label{eq:maximal_estimates2}
\lim_{\lambda\to\infty}\lambda \mm\bigl(\{M_\infty g>\lambda\}\bigr)=0.
\end{equation}
This asymptotic version follows by \eqref{eq:maximal_estimates1}, taking the inclusion 
$$\{M_\infty g>2\lambda\}\subset \{M_\infty(g-\lambda)^+>\lambda)\}$$ into account.

Using maximal functions, and under the doubling assumption, one can prove the existence of approximate limits
of functions $f\in L^1(X,\mm)$: the approximate limit $\tilde{f}(x)$ at $x$ is defined by the
property
$$
\lim_{r\downarrow 0}\media_{B_r(x)}|f(y)-\tilde{f}(x)|\,d\mm(y)=0.
$$
Points where the approximate limit exists are called approximate continuity points.
It turns out that the approximate limit exists $\mm$-a.e. in $X$ and that $\tilde{f}(x)=f(x)$ for
$\mm$-a.e. $x\in X$ 
(notice that the first function is pointwise defined in its domain, while the latter belongs to a
Lebesgue equivalence class).

\subsection{Pseudo gradients and Cheeger energies}\label{sec:Chee}\label{ssec3}

In this section we recall the basic facts of the theory of relaxed gradients. As we will see, even though the
initial class of functions is a priori chosen to be the class of locally Lipschitz functions, it is sometimes 
technically useful to consider objects different from the slope in the relaxation process; for instance
in Cheeger's paper \cite{Chee} upper gradients were used; while in \cite[Section~8.3]{AGS12} and \cite{AmCoDi} the so-called asymptotic
Lipschitz constant was used, and proved to be technically useful. See also \cite{GoTr} for a closely related axiomatization.

\begin{definition}[Pseudo gradient]
We call a function $G$ on ${\rm Lip}_{{\rm loc}}(X)$ with values into nonnegative Borel functions 
a \emph{pseudo gradient} if the following
properties hold:
\begin{itemize}
\item[(i)] $G$ is positively $1$-homogeneous, i.e. $G(tf)=tG(f)$ for all $t\geq 0$;
\item[(ii)] $G$ is convex;
\item[(ii)] for some constant $C$, 
$$
G(\chi u+(1-\chi)v)\leq \chi G(u)+(1-\chi)G(v)+C{\rm Lip}(\chi)|u-v|
$$
for all $u,\,v,\,\chi\in{\rm Lip}_{{\rm loc}}(X),\,\,0\leq\chi\leq 1$.
\end{itemize}
\end{definition}

The main example of pseudo gradient is the slope in \eqref{def:slope}, other examples are for instance
the one-sided slopes
$$
|\nabla^+f|(x)=\limsup_{y\to x}\frac{(f(y)-f(x))^+}{d(y,x)},\qquad
|\nabla^-f|(x)=\limsup_{y\to x}\frac{(f(y)-f(x))^-}{d(y,x)},
$$
or the cartesian slope in product spaces, see \eqref{eq:cartesian_slope} below.

Given a pseudo gradient $G$, the Cheeger energy $\C_G:L^2(X,\mm)\to [0,\infty]$
associated to $G$ is defined by
$$
\C_G(f):=\inf\left\{\liminf_{h\to\infty}\int_XG^2(f_h)\,d\mm:\ \text{$f_h\in {\rm Lip}_{\rm loc}(X)$,  $f_h\to f$ in $L^2(X,\mm)$}\right\},
$$
with the convention $\inf\emptyset=\infty$.
We shall denote by $\C$ the ``canonical'' Cheeger energy associated to the slope and define
$$
H^{1,2}(X,d,\mm):=\left\{f\in L^2(X,\mm):\ \C(f)<\infty\right\}.
$$

We call $u\in L^2(X,\mm)$ a $G$-relaxed slope of $f$ if there exists Lipschitz functions
$f_n$ satisfying $f_n\to f$ strongly in $L^2(X,\mm)$ and $G(f_n)\to v$ weakly in $L^2(X,\mm)$, with $v\leq u$ $\mm$-a.e. in $X$.

The following results collect the main facts about Cheeger energies and $G$-relaxed slopes; although
the results in \cite{Chee} and \cite{AGS11a} are not stated in terms of pseudo-gradients, their proof extend
with no change to this more general framework.

\begin{theorem}\label{thm:mainchee} For any $f\in L^2(X,\mm)$ and any
pseudo gradient $G$ the following properties hold.
\begin{itemize}
\item[(i)] The collection of $G$-relaxed slopes of $f$ is a convex closed set, possibly empty.
\item[(ii)] If the collection of $G$-relaxed slope is not empty, its element with minimal norm $|\nabla f|_{*,G}$
satisfies
\begin{equation}\label{eq:minimal}
\text{$|\nabla f|_{*,G}\leq u$ $\mm$-a.e.  in $X$ for any $G$-relaxed slope $u$.}
\end{equation}
Furthermore, $|\nabla f|_{*,G}$ can be obtain as the strong $L^2(X,\mm)$ limit of a sequence $G(f_n)$, with
$f_n$ locally Lipschitz and $f_n\to f$ in $L^2(X,\mm)$. 
\item[(iii)] $\sqrt{\C_G}$ is a convex lower semicontinuous functional in $L^2(X,\mm)$, positively
1-homogeneous and even if $G$ is even. In addition, $\C_G(f)<\infty$ if 
and only if the collection of $G$-relaxed slopes is not empty. In this case,
$$
\C_G(f)=\int_X |\nabla f|_{*,G}^2\,d\mm.
$$
\end{itemize}
\end{theorem}

We call $|\nabla f|_{*,G}$ the minimal relaxed slope and, as we did for $\C$, denote by $|\nabla f|_*$, the
``canonical'' one associated to $G(f)=|\nabla f|$. 

Recall that the subdifferential of a convex lower semicontinuous functional $\Phi:H\to (-\infty,+\infty]$ in a Hilbert space is 
defined, at any point $u$ such that $\Phi(u)<\infty$, by
$$
\partial\Phi(u):=\left\{\xi\in H:\ \Phi(v)\geq\Phi(u)+\langle\xi,v-u\rangle\,\,\forall v\in H\right\}.
$$
Now, denoting by $\Delta_G$ the element with minimal norm in the subdifferential of $\C_G$,
one can use the inequality 
$$
|\nabla (f+\epsilon g)|_{*,G}\leq |\nabla f|_{*,G}+\epsilon|\nabla g|_{*,G}
\qquad\text{$\mm$-a.e. in $X$, for all $\epsilon>0$,}
$$
which is a simple consequence of the convexity and homogeneity of $G$,
to prove the following integration by parts formula, see for instance \cite[Proposition~4.15]{AGS11a} (the
second part of the proposition, dealing with the equality cases, uses also the chain rule \cite[Proposition~4.8]{AGS11a}).

\begin{proposition} For all $u\in D(\C_G)$ and all $v\in D(\C_G)$ such that $\partial \C_G(v)$ is not empty, one has
\begin{equation}\label{eq:bypartsG}
-\int u\Delta_G v\,d\mm\leq \int |\nabla u|_{*,G}|\nabla v|_{*,G}\,d\mm,
\end{equation}
with equality if $u=\phi(v)$ with $\phi:\R\to\R$ Lipschitz, continuously differentiable and nonincreasing.
\end{proposition}

In Section~\ref{sec:per} we will need to consider the Cheeger energy and associated weak gradient for a measure finite on bounded sets; these are defined in the same way as for the finite case. Locality of the weak gradients allows us to generalize statements about pointwise behaviour also to this case
(arguing as in ). The assumption that $\mm$ is finite is not necessary in the following fundamental result obtained, among other things, in \cite{Chee}.

\begin{theorem}[Minimal relaxed slope coincides $\mm$-a.e. with the slope]\label{thm:cheeger}
Assume that $(X,d,\mm)$ is doubling and that $(1,2)$-Poincar\'e holds for locally Lipschitz functions. Then
$$
|\nabla f|=|\nabla f|_*\qquad\text{$\mm$-a.e. in $X$, for all $f\in {\rm Lip}_{\rm loc}(X)\cap H^{1,2}(X,d,\mm)$.}
$$
\end{theorem}

On $X=\R^d$ endowed with the distance induced by a norm $\|\cdot\|$ and with the Lebesgue measure $\mm=\Leb{d}$, it is
easily seen that (for $f$ smooth or locally Lipschitz)
\begin{equation}\label{eq:gig1}
\C(f)=\int_X \|\nabla f\|_*^2\,d\mm,
\end{equation}
where $\|\cdot\|_*$ is the dual norm and $\nabla f$ is the canonical Euclidean gradient.
Hence, $\C$ is a quadratic form if and only if $\|\cdot\|$ is induced by a scalar product. As in \cite{Gigli12},
this motivates two possible definition of spaces which are ``Hilbertian on small scales''.

\begin{definition}[Strongly asymptotically Hilbertian m.m.s.]
We say that $(X,d,\mm)$ is strongly asymptotically Hilbertian if $f\mapsto \int_X|\nabla f|^2\,d\mm$ is a quadratic
form on ${\rm Lip}_{\rm loc}(X)$.
\end{definition}

\begin{definition}[Asymptotically Hilbertian m.m.s.]
We say that $(X,d,\mm)$ is asymptotically Hilbertian if $\C$ is a quadratic form on $L^2(X,\mm)$.
\end{definition}

\begin{proposition} Any strongly asymptotically Hilbertian space is asymptotically Hilbertian. If doubling and
$(1,2)$-Poincar\'e inequality hold, then the two properties are equivalent.
\end{proposition}
\begin{proof}
The first statement is a simple consequence of the fact that the lower semicontinuous relaxation of
a quadratic form is still a quadratic form. The second one follows by Theorem~\ref{thm:cheeger}.
\end{proof}

\subsection{Quadratic forms}

Let $H$ be a separable Hilbert space and let $Q:H\to [0,\infty]$ be a lower semicontinuous quadratic form. We shall denote by
$A$ the associated bilinear form on $D(Q)=\{Q<\infty\}$, namely
$$
A(u,v)=\frac{1}{4}\bigl(Q(u+v)-Q(u-v)\bigr)\qquad u,\,v\in D(Q).
$$
Assuming $D(Q)$ to be dense in $H$, we shall also denote by $L$ the possibly unbounded operator whose domain $D(L)$ consists of all $u\in D(Q)$
satisfying
\begin{equation}\label{eq:Bassano}
A(u,v)=\langle w,v\rangle\qquad\forall v\in D(Q)
\end{equation}
for some $w\in H$.
Since the density of $D(Q)$ in $H$ ensures that $w$ is uniquely determined by \eqref{eq:Bassano} we can set $w=Lu$, so that
\begin{equation}\label{eq:Bassano1}
A(u,v)=\langle Lu,v\rangle\qquad u\in D(L),\,v\in D(Q).
\end{equation}

The following classical result is the spectral theorem for compact operators, stated with assumptions at the level of
the quadratic form.

\begin{theorem}[Spectral theorem]\label{thm:Bassano}
Let $Q:H\to [0,\infty]$ be a lower semicontinuous quadratic form with dense domain. Assume the existence
of $c>0$ satisfying $Q(u)\geq c\|u\|^2$ for all $u\in H$, and that the sublevel sets
$\bigl\{u\in H:\ Q(u)\leq M\bigr\}$
are compact in $H$ for all $M\geq 0$. Then there exists a complete orthonormal basis of $H$ made of eigenvectors
of $L$.
\end{theorem}
\begin{proof} It is simple to check, by the minimization of $v\mapsto \tfrac 12 Q(v)-\langle f,v\rangle$, that $L:D(L)\to H$
is onto. In addition,
$$
c\|u\|^2\leq Q(u)=A(u,u)=\langle Lu,u\rangle \leq \|u\|\|Lu\|\qquad\forall u\in D(L)
$$
shows that $L^{-1}$ is continuous, with $\|L^{-1}\|\leq 1/c$.
In order to apply the classical spectral theorem for compact operators to $L^{-1}$ (see for instance \cite[Theorem~VI.11]{Bre}), 
we need only to show that any sequence $(u_n)\subset D(L)$
such that $(\|Lu_n\|)$ is bounded is relatively compact in $H$. To this aim, notice that the previous inequality yields
$Q(u_n)\leq \|L(u_n)\|^2/c$, so that $Q(u_n)$ is bounded and $(u_n)$ is relatively compact in $H$.
\end{proof}

\subsection{$BV$ functions}\label{ssec4}

We recall the definition of $BV$ introduced in \cite{Mi} for locally compact spaces; this class
has been further studied in \cite{AmDi}, dropping the local compactness assumption, and it  
has been characterized by the behaviour of the function along curves, see Theorem~\ref{thm:amdm} below. 

\begin{definition}[The space $BV(X,d,\mm)$] \label{def:myBV} Let $f\in L^1(X,\mm)$. We say that $f$ belongs to $BV(X,d,\mm)$
if there exist locally Lipschitz functions $f_n:X\to\R$ such that $f_n\to f$ in $L^1(X,\mm)$ and $\limsup_n\int_X|\nabla f_n|\,d\mm<\infty$.
\end{definition}

For $f\in BV(X,d,\mm)$ and $A\subset X$ open, we may consider the set function
\begin{equation}\label{eq:defDu}
|Df|(A):=\inf\left\{\liminf_{n\to\infty}\int_A|\nabla f_n|\,d\mm:\ f_n\in{\rm Lip}_{\rm loc}(A),\,\,\,\lim_{n\to\infty}\int_A|f_n-f|\,d\mm=0
\right\}.
\end{equation}
It can be proved that $A\mapsto |Df|(A)$ is the restriction to open sets of a finite Borel measure (it is
in this proof that it turns out to be useful to consider {\it locally} Lipschitz functions in Definition~\ref{def:myBV} and \eqref{eq:defDu}), that we still
denote by $|Df|$. Notice also the elementary inequality (first proved on open sets, and then extended to Borel
sets)
\begin{equation}\label{eq:subadd}
|D(f+g)|\leq |Df|+|Dg|\qquad\forall f,\,g\in BV(X,d,\mm).
\end{equation}
By construction, we have the lower semicontinuity property
$$
\liminf_{n\to\infty}|Df_n|(X)<\infty\quad\Rightarrow\quad\text{$f\in BV(X,d,\mm)$ and $|Df|(X)\leq\liminf_{n\to\infty}|Df_n|(X)$}
$$
for $f_n\in BV(X,d,\mm)$ convergent to $f$ in $L^1(X,\mm)$.

\begin{remark}[Sets of finite perimeter in locally finite m.m.s.]\label{rem:setsfinp}{\rm
When concerned with characteristic functions, we will apply these concepts also in 
m.m.s. whose reference measure $\mm$ is finite on bounded sets, i.e. we weaken 
the finiteness assumption on $\mm$ by requiring that $\mm(B)<\infty$ for any
bounded Borel set $B$ (specifically, in Section~\ref{sec:per} we will consider the product of a finite m.m.s. with the real line).\\ We say that $E\subset X$ has finite
perimeter in $X$ if the characteristic function $\chi_E$ belongs to $BV(Y,d,\mm)$ for any closed subset $Y\subset X$
with finite measure and $\sup_n|D\chi_E|(B_n(x_0))<\infty$ for some (and thus all) $x_0\in X$. 
By monotone approximation with open sets with finite measure, still $|D\chi_E|(B)$ is well defined for any open set $B\subset X$ and
$B\mapsto |D\chi_E|(B)$ is a positive finite Borel measure in $X$. Furthermore, the monotonicity of the approximation gives that 
$E\mapsto |D\chi_E|(A)$ is lower semicontinuous with respect to local convergence in $\mm$-measure (i.e.
$L^1_{\rm loc}$ convergence of the characteristic functions) for any open set $A\subset X$. 

When dealing with sets of finite perimeter, we will use the traditional notation
$$
P(E,B):=|D\chi_E|(B).
$$\fr}
\end{remark}

By a diagonal argument in \eqref{eq:defDu}, it is clear that there exist $f_n$ locally Lipschitz convergent to $f$ in $L^1(X,\mm)$
and satisfying $\limsup_n\int_X|\nabla f_n|\,d\mm\leq |Df|(X)$. Since, by the very definition of $|Df|(A)$, it holds
$$
\liminf_{n\to\infty}\int_A|\nabla f_n|\,d\mm\geq |Df|(A)\qquad\text{for any open set $A\subset X$,}
$$
a well known criterion for weak convergence of measures gives the approximation property:
\begin{equation}\label{optimal_sequence_BV}
\forall f\in BV(X,d,\mm)\,\,\,\exists f_n\in {\rm Lip}_{\rm loc}(X)\quad\text{with}\quad
|\nabla f_n|\mm\rightharpoonup |Df|\quad\text{in duality with $C_b(X)$.}
\end{equation}

We can now investigate the implications of doubling and $(1,1)$-Poincar\'e for Lipschitz functions. From \eqref{optimal_sequence_BV} we immediately obtain the
$BV$ version of the Poincar\'e inequality in $BV$, namely 
$$
\media_{B_r(x)}| f-f_{x,r}|\,d\mm\leq c r |Df|\bigl(\overline{B}_{\lambda r}(x)\bigr)\qquad\forall f\in BV(X,d,\mm).
$$
Possibly replacing $\lambda$ by $\lambda'>\lambda$, we will keep using the traditional form 
\begin{equation}\label{eq:PIBV}
\media_{B_r(x)}| f-f_{x,r}|\,d\mm\leq c r |Df|\bigl(B_{\lambda r}(x)\bigr)\qquad\forall f\in BV(X,d,\mm),
\end{equation}
stated with open sets. This estimate, when combined with the doubling property of $(X,d,\mm)$, leads by standard arguments
(see for instance \cite{Chee}, \cite{AmCoDi}) to the following proposition.

\begin{proposition} Assume that $(X,d,\mm)$ is doubling and that $(1,1)$-Poincar\'e holds for Lipschitz functions. 
Then, there exists
a structural constant $c>0$ satisfying
\begin{equation}\label{eq:liploc}
|\tilde f(x)-\tilde f(y)|\leq c d(x,y)\bigl( M_r|Df|(x)+M_r|Df|(y)\bigr) 
\end{equation}
whenever $f\in BV(X,d,\mm)$, $x$ and $y$ are approximate continuity points of $f$, $r>0$ and $d(x,y)<r/c$.
\end{proposition}

\subsection{Equivalence of weak gradients and locality}\label{ssec5}

We denote by $C([0,1];X)$ the space of continuous maps from $[0,1]$ to $X$ endowed with the standard
Polish structure and denote by $e_t:C([0,1];X)\to X$, $t\in [0,1]$, the evaluation maps at time $t$, namely $e_t(\gamma)=\gamma(t)$.
We shall also denote by $(e_t)_\#\ppi$ the push-forward probability measure on $X$ induced by $\ppi\in\Probabilities{C([0,1];X)}$, namely
$$
\int_X \phi\,d(e_t)_\#\ppi=\int \phi(\gamma(t))\,d\ppi(\gamma)
$$
for any $\phi:X\to\R$ bounded Borel, or Borel nonnegative.

We say that $\ppi\in\Probabilities{C([0,1];X)}$ is a $\infty$-test plan if it is concentrated on Lipschitz curves, 
${\rm Lip}(\gamma)\in L^\infty(C([0,1];X),\ppi)$ and the non-concentration property
\begin{equation}\label{eq:bounded_compression}
(e_t)_\sharp\ppi\leq C\mm\qquad\forall t\in [0,1]
\end{equation}
holds for some constant $C\geq 0$. We shall denote by $C(\ppi)$ the least constant satisfying \eqref{eq:bounded_compression}.

We say that a Borel family $\Gamma\subset C([0,1];X)$ is
$1$-negligible if $\ppi(\Gamma)=0$ for any $\infty$-test plan $\ppi$. This notion is weaker than the non-parametric
notion of ${\rm Mod}_1$-negligible set. Recall that $\Gamma$ is said to be ${\rm Mod}_1$-negligible if for any
$\eps>0$ there exists a Borel function $\rho:X\to [0,\infty]$ with $\int_X\rho\,d\mm<\eps$ and $\int_\gamma\rho\geq 1$
for all $\rho\in\Gamma$. Since
$$
\ppi(\Gamma)\leq\int\int_\gamma \rho\,d\ppi(\gamma)\leq \|{\rm Lip}(\gamma)\|_\infty\int_0^1\int \rho(\gamma(t))\,d\ppi(\gamma)\,dt
\leq  \|{\rm Lip}(\gamma)\|_\infty C(\ppi)\int_X\rho\,d\mm
$$
for any $\infty$-test plan $\ppi$, we obtain that ${\rm Mod}_1$-negligible sets are $1$-negligible. See \cite{AmDiSa}
for a much more detailed comparison between notions of negligibility for families of curves, both parametric and
non-parametric.

The next theorem is one of the main results of \cite{AmDi}.

\begin{theorem}\label{thm:amdm}
Let $f\in L^1(X,d,\mm)$. Then $f\in BV(X,d,\mm)$ if and only if there exists a finite Borel measure $\mu$ on $X$ with the 
following property: for any $\infty$-test plan $\ppi$ one has:
\begin{itemize}
\item[(i)] $f\circ\gamma\in BV(0,1)$ for $\ppi$-a.e. $\gamma$;
\item[(ii)] $\int \gamma_\sharp|D(f\circ\gamma)|\,d\ppi(\gamma)\leq C(\ppi)\|{\rm Lip}(\gamma)\|_\infty \mu$.
\end{itemize}
For all $f\in BV(X,d,\mm)$, the smallest measure $\mu$ satisfying (ii) is $|Df|$.
\end{theorem}

Again, in (ii) by $\gamma_\sharp|D(f\circ\gamma)|$ we mean the push-forward measure defined by
$$
\int_X \phi\, d\gamma_\sharp|D(f\circ\gamma)|=\int_0^1 \phi\circ\gamma \, d |D(f\circ\gamma)|.
$$
Notice
that (i) and (ii) are invariant in the Lebesgue equivalence class of $f$: indeed, because of \eqref{eq:bounded_compression}, 
$\ppi$-a.e. curve $\gamma$ hits a prescribed $\mm$-negligible set in a $\Leb{1}$-negligible set of times, so that the integral in (ii)
does not change if we replace $f$ by $\hat{f}$, with $\mm(\{f\neq\hat{f}\})=0$.

\begin{remark}\label{rem:bad_things_occur} {\rm 
Let us split $|D(f\circ\gamma)|$ into absolutely continuous $|D^a (f\circ\gamma)|$ and singular
$|D^s(f\circ\gamma)|$ parts w.r.t. $\Leb{1}$.
Notice that without further assumptions on the metric measure structure, while we know that
$$
\int \gamma_\sharp|D^a(f\circ\gamma)|\,d\ppi(\gamma)\ll\mm
$$
(again because $\ppi$-a.e. curve $\gamma$ hits a $\mm$-negligible set
in a $\Leb{1}$-negligible set of times), we can't say in general that
$$
\int \gamma_\sharp|D^s(f\circ\gamma)|\,d\ppi(\gamma)
$$
is singular w.r.t. $\mm$. See \cite[Example~7.4]{AmDi} for an explicit example of a function $f\in BV(X,d,\mm)$ (actually
a characteristic function) with $|Df|\leq C\mm$, but $|D(f\circ\gamma)|$ singular w.r.t. $\Leb{1}$ for any curve $\gamma$.
}\end{remark}

\section{Tensorization of metric measure spaces}\label{sec:tenso}

In this section, as in Section~6 of \cite{AGS11b}, we consider two complete and separable m.m.s. $(X,d_X,\mm_X)$, $(Y,d_Y,\mm_Y)$ and their product
$Z=X\times Y$, endowed with the product distance $d$ satisfying $d^2=d_X^2+d_Y^2$ and the product measure $\mm$. We assume $\mm_X$ and
$\mm_Y$ to be finite. We denote by
$\C$ the Cheeger energy in $(Z,d,\mm)$ and use the notation $f^x(y)=f(x,y)$, $f^y(x)=f(x,y)$.

In Section 6 of \cite{AGS11b}, it is proved, under curvature assumptions on the factors, that $\C$ is quadratic whenever
$\C_X$ and $\C_Y$ are quadratic. In addition, defining
\begin{equation}\label{domainJ}
\Lambda:=\bigl\{f\in L^2(Z,\mm):\ \begin{cases}&\text{$f^x\in H^{1,2}(Y,d_Y,\mm_Y)$ for $\mm_X$-a.e. $x\in X$}\\
&\text{$f^y\in H^{1,2}(X,d_X,\mm_X)$ for $\mm_Y$-a.e. $y\in Y$}\end{cases}\bigr\}, 
\end{equation}
we have the inclusion $H^{1,2}(Z,d,\mm)\subset\Lambda$ and any
$f\in H^{1,2}(Z,d,\mm)$ satisfies
\begin{equation}\label{eq:factor}
|\nabla f|^2_*(x,y)=|\nabla f^y|^2_{*,X}(x)+|\nabla f^x|^2_{*,Y}(y)
\qquad\text{for $\mm$-a.e. $(x,y)\in Z$.}
\end{equation}

Notice that \eqref{eq:factor} provides immediately that $\C$ is a quadratic form whenever $\C_X$ and $\C_Y$
are quadratic forms, since 
$$
\C(f)=\int_Y\C_X(f^y)\,d\mm_Y(y) + \int_X\C_Y(f^x)\,d\mm_X(x).
$$
Notice also that, as illustrated by Proposition~\ref{prop:easy_tensorization} below, 
\eqref{eq:factor} might be true independently of quadraticity assumptions on $\C_X$ and $\C_Y$, just by ``duality''
with the formula 
$d^2=d_X^2+d_Y^2$ defining $d$, but the full validity of this dual formula is presently an open problem (more precisely,
the general validity of $\leq$ in \eqref{eq:factor} is not known, while we will prove that the converse inequality always holds).
 
\begin{proposition}[An easy case of tensorization]\label{prop:easy_tensorization}
Assume that $X$ and $Y$ are Euclidean spaces, with distances $d_X$ and $d_Y$ induced by norms $\|\cdot\|_X$ and
$\|\cdot\|_Y$ respectively. If $\mm_X$ and $\mm_Y$ are the corresponding Lebesgue measures, then \eqref{eq:factor} holds.
\end{proposition}
\begin{proof} The norm $\|\cdot\|_Z$ in $Z=X\times Y$ corresponding to the distance $d_Z=\sqrt{d_X^2+d_Y^2}$ obviously 
satisfies $\|(x,y)\|^2_Z=\|x\|_X^2+\|y\|_Y^2$. According to \eqref{eq:gig1}, we have to prove that 
$$
\|(x^*,y^*)\|_{*,Z}^2=\|x^*\|_{*,X}^2+\|y^*\|_{*,Y}^2\qquad\forall x^*\in X^*,\,\, y^*\in Y^*,
$$
where $\|\cdot\|_{*,Z}$ denotes the dual norm of $\|\cdot\|_Z$ and $\|\cdot\|_{*,X}$ and $\|\cdot\|_{*,Y}$ denote
the dual norms in $X$ and $Y$ respectively. We can estimate
\begin{eqnarray*}
\langle (x^*,y^*),(x,y)\rangle&=&\langle x^*,x\rangle+\langle y^*,y\rangle\\&\leq& \|x^*\|_{*,X}\|x\|_X+\|y^*\|_{*,Y}\|y\|_Y\\&\leq&
\sqrt{\|x^*\|_{*,X}^2+\|y^*\|^2_{*,Y}}\|(x,y)\|_Z.
\end{eqnarray*}
This proves that $\|(x^*,y^*)\|^2_{*,Z}\leq\|x^*\|_{*,X}^2+\|y^*\|_{*,Y}^2$.
On the other hand, if we choose nonzero vectors $x$, $y$ such that $\langle x^*,x\rangle=\|x^*\|_{*,X}\|x\|_X$ and
$\langle y^*,y\rangle=\|y^*\|_{*,Y}\|y\|_X$ we obtain that the first inequality is an equality. Then, replacing $x$ by $ax$
and $y$ by $by$ with $a,\,b\geq 0$ in such a way that
$$
\|x\|_X=\|x^*\|_{*,X}\quad\text{and}\quad\|y\|_Y=\|y^*\|_{*,Y}
$$  
we retain the first two equalities and we obtain that also the last inequality is an equality.
\end{proof}

In Section 6 of \cite{AGS11b}, through a detailed analysis of the Hopf-Lax formula, 
inequality $\geq$ in \eqref{eq:factor} is proved independently of curvature assumptions, while curvature enters in the proof of the
converse one (see \cite[Lemma~6.15]{AGS11b}).

In this section we will revise carefully the arguments of \cite{AGS11b}, obtaining some refinements independently of curvature
assumptions.

We denote by $|\nabla f|_c:{\rm Lip}_{\rm loc}(Z)\to [0,\infty]$ the convex pseudo gradient
\begin{equation}\label{eq:cartesian_slope}
|\nabla f|_c(x,y):=\sqrt{|\nabla f^y|^2(x)+|\nabla f^x|^2(y)},
\end{equation}
corresponding to a ``cartesian'' slope. Notice that 
\begin{equation}\label{eq:bslope}
|\nabla f|(x,y)\geq\max\{|\nabla f^x|(y),|\nabla f^y|(x)\}\geq\frac 12 |\nabla f|_c(x,y).
\end{equation}

We denote by $\C_c$ the Cheeger energy arising from the $L^2$ relaxation of $\int |\nabla f|_c^2\,d\mm$ restricted to locally
Lipschitz functions, namely
$$
\C_c(f):=\inf\left\{\liminf_{n\to\infty}\int_Z|\nabla f_n|_c^2\,d\mm:\ f_n\in {\rm Lip}_{{\rm loc}}(Z),\,\,\int_Z|f_n-f|^2\,d\mm\rightarrow 0\right\}. 
$$
As we discussed in Section~\ref{sec:Chee}, we have a minimal relaxed gradient associated to any $f\in D(\C_c)$, that we shall denote by
$|\nabla f|_{*,c}$. It is also clear from \eqref{eq:bslope} that $\C_c\leq 2\C$. Even though it seems difficult to establish a pointwise converse inequality
$$|\nabla f|(x,y)\leq C\max\{|\nabla f^x|(y),|\nabla f^y|(x)\}\leq C|\nabla f|_c$$
for some constant $C$, we are able to refine a bit the scheme of \cite{AGS11b} to prove
the following result, independently of doubling, quadraticity and Poincar\'e assumptions. 
\begin{theorem} \label{thm:I=C} $\C_c=\C$ on $L^2(Z,\mm)$ and
\begin{equation}\label{eq:factor2}
|\nabla f|_*=|\nabla f|_{*,c}\quad\text{$\mm$-a.e. in $Z$, for all $f\in H^{1,2}(Z,d,\mm)$.}
\end{equation}
\end{theorem}
\begin{proof} We provide the proof of the inequality $\C_c\leq\C$ in the appendix, since it involves tools (Hopf-Lax formula, gradient flows)
not directly connected to the rest of the paper. Let us prove the inequality $\C\leq\C_c$ and the corresponding inequality
$\leq$ in \eqref{eq:factor2}.  By Lemma~6.14 in \cite{AGS11b} (in turn based on the calculus Lemma~4.3.4 of \cite{AGS05}), 
$|\nabla f|_c$ is an upper gradient for locally Lipschitz functions, namely
\begin{equation}\label{eq:slopecart}
|\frac{d}{dt} (f\circ\gamma)|\leq |\nabla f|_c\circ\gamma|\dot\gamma|\quad\text{$\Leb{1}$-a.e. in $(0,1)$}
\end{equation}
for all $\gamma\in AC([0,1];Z)$ and $f\in {\rm Lip}_{\rm loc}(Z)$. Recall now that for any $f\in D(\C_c)$ we can find locally
Lipschitz functions
$f_n$ with $f_n\to f$ in $L^2(Z,\mm)$ and $|\nabla f_n|_c\to |\nabla f|_{*,c}$ in $L^2(Z,\mm)$.
Then, the standard argument for the closure of weak upper gradients (even under weak $L^2(Z,\mm)$ convergence, see
\cite[Lemma~4.11]{Sh}) provides the inclusion $D(\C)\supset D(\C_c)$ and the inequality
\begin{equation}\label{eq:cdc1}
|\nabla f|_*\leq |\nabla f|_{*,c}\qquad\text{$\mm$-a.e. in $Z$.} 
\end{equation}
\end{proof}

It will be useful to compare $\C_c$ with the convex functional (actually a quadratic form, when 
$\C_X$ and $\C_Y$ are quadratic)
\begin{equation}\label{eq:def_J}
J(f):=\int_Z |\nabla f^y|_{*,X}^2(x)+|\nabla f^x|_{*,Y}^2(y)\,d\mm(x,y).
\end{equation}
The functional $J$ is well defined on the set $\Lambda$ in \eqref{domainJ}, and set equal to $\infty$ on $L^2(Z,\mm)\setminus\Lambda$.

\begin{lemma}\label{lem2}
If $f\in L^2(Z,\mm)$ is representable as $f_Xf_Y$, with $f_X\in H^{1,2}(X,d_X,\mm_X)$ and $f_Y\in H^{1,2}(Y,d_Y,\mm_Y)$, then $\C_c(f)=J(f)$.
Furthermore, if $-\xi_X\in\partial\C_X(f_X)$ and $-\xi_Y\in\partial\C_Y(f_Y)$, then
\begin{equation}\label{eq:partial_lap2}
-(f_Y\xi_X+f_X\xi_Y)\in\partial J\bigl(f_Xf_Y).
\end{equation} 
\end{lemma}
\begin{proof} The inequality $\C_c(f)\geq J(f)$, even for all functions $f$, is guaranteed by the previous lemma.
In order to prove the converse inequality $\C_c(f)\leq J(f)$ on tensor products $f$, 
by a simple truncation argument we can assume with no loss of generality 
that $f_X$ and $f_Y$ are bounded functions.
Let $f^n_X\in{\rm Lip}_{\rm loc}(X)$ with $|\nabla f^n_X|\to |\nabla f_X|_{*,X}$ in $L^2(X,\mm_X)$; analogously, let
$f^n_Y\in{\rm Lip}_{\rm loc}(Y)$ with $|\nabla f^n_Y|\to |\nabla f_Y|_{*,Y}$ in $L^2(Y,\mm_Y)$. Since $f_X$ and $f_Y$
are bounded, we can also assume that $(f^n_X)$ and $(f^n_Y)$ are uniformly bounded. Since $f^n_Xf^n_Y\in {\rm Lip}_{\rm loc}(Z)$ and
$$
|\nabla (f^n_Xf^n_Y)|_c^2=(f^n_X)^2|\nabla f_Y^n|^2+(f^n_Y)^2|\nabla f_X^n|^2
$$
we can easily conclude.

The verification of \eqref{eq:partial_lap2} is straightforward.
\end{proof}

Now we can state the tensorization of doubling metric measure spaces satisfying a $(1,2)$-Poincar\'e inequality for locally
Lipschitz functions. Note that, by a localization argument, the following theorem is also valid for measures finite on bounded sets.

\begin{theorem} \label{thm:PI_tensor} Assume that $(X,d_X,\mm_X)$ and $(Y,d_Y,\mm_Y)$ are doubling and satisfy a 
$(1,2)$-Poincar\'e inequality for locally Lipschitz functions. Then the tensorization property \eqref{eq:factor} of weak gradients
holds for all functions $f\in D(\C)$. In addition, $\C$ concides with the functional $J$ defined in \eqref{eq:def_J}.
\end{theorem}

Our third result on the tensorization provides the quadratic property of $\C$ independently of curvature assumptions, but assuming the
strong asymptotically Hilbertian property on the factors. As we discussed in Section~\ref{sec:Chee}, this assumption is equivalent
to asymptotic Hilbertianity in the presence of doubling and $(1,2)$-Poincar\'e inequality for locally Lipschitz functions, which are
already covered by Theorem~\ref{thm:PI_tensor}. Again, by a localization argument, the following theorem also holds for measures
finite on bounded sets.

\begin{theorem} \label{thm:Cquad} If $\C_X,\,\C_Y$ are quadratic, then \eqref{eq:factor} holds iff $\C$ is quadratic. 
In particular, if $(X,d_X,\mm_X)$ and $(Y,d_Y,\mm_Y)$ are
strongly asymptotically Hilbertian, then $\C$ is a quadratic form and \eqref{eq:factor} holds.
\end{theorem}

We devote the rest of the section to the proof of Theorem~\ref{thm:PI_tensor} and Theorem~\ref{thm:Cquad}. 

\begin{lemma}\label{lem1}
$J$ is a convex and weakly lower semicontinuous functional on $L^2(Z,\mm)$. In addition,
$\C_c\geq J$ on $L^2(Z,\mm)$ and $|\nabla f^x|_{*,Y}^2(y)+|\nabla f^y|_{*,X}^2(x)\leq |\nabla f|_{*,c}^2(x,y)$ for $\mm$-a.e. $(x,y)\in Z$.
\end{lemma}
\begin{proof} Convexity of $J$ is trivial, so it suffices to show lower semicontinuity of $J$ in the strong topology of $L^2(Z,\mm)$.
This and the last part of the statement are a simple application of Fubini's theorem and of the fact that any convergent sequence 
$g_n\to g$ in $L^2(Z,\mm)$ can be refined to obtain a subsequence $(n_k)$ satisfying
$g^x_{n(k)}\to g^x$ in $L^2(Y,\mm_Y)$ for $\mm_X$-a.e. $x\in X$ and 
$g^y_{n(k)}\to g^y$ in $L^2(X,\mm_X)$ for $\mm_Y$-a.e. $y\in Y$.
\end{proof}

\begin{proof}[of Theorem~\ref{thm:PI_tensor}] 
Taking Lemma~\ref{lem1}, \eqref{eq:factor2} and the equality $\C=\C_c$ into account, to prove
the tensorization property \eqref{eq:factor} for all functions $f\in D(\C)$ we need only to show
that $\C_c\leq J$ on $H^{1,2}(Z,d,\mm)$ and that $D(J)\subset H^{1,2}(Z,d,\mm)$. To prove the inequality we need to find, for all $f\in H^{1,2}(Z,d,\mm)$ a
sequence of functions $f_n\in {\rm Lip}_{\rm loc}(Z)\cap L^2(Z,\mm)$ satisfying $f_n\to f$ in $L^2(Z,\mm)$ and $J(f_n)\to J(f)$. Indeed,
since Theorem~\ref{thm:cheeger} gives
$$
\C_c(f_n)\leq \int_Z|\nabla f_n|^2_c\,d\mm=J(f_n),
$$
the lower semicontinuity of $\C_c$ provides the result.

Since $(Z,d,\mm)$ is a doubling space satisfying the $(1,2)$-Poincar\'e inequality on Lipschitz functions, 
in order to build $f_n$ we can use the inequality (see for instance \cite[Lemma~46]{AmCoDi})
$$
|\tilde f(z)-\tilde f(z')|\leq C\bigl((M|\nabla f|_*^2)^{1/2}(z)+(M|\nabla f|_*^2)^{1/2}(z')\bigr) d(z,z')
$$
valid at all approximate continuity of $f$, and the sets 
$$
E_n:=\left\{z\in Z:\ \max\{|\tilde f(z)|,(M|\nabla f|_*^2)^{1/2}(z)\}\leq n\right\}.
$$
By the MacShane lemma we can extend the restriction of $\tilde f$ to $E_n$ to a Lipschitz function $f_n$ with
${\rm Lip}(f_n)\leq Cn$ and $|f_n|\leq n$. Then, a standard argument based on maximal inequalities
shows that $f_n\to f$ in $L^2(Z,\mm)$, $\mm(Z\setminus E_n)\downarrow 0$ and $|\nabla f_n|^2$
are equi-integrable in $L^1(Z,\mm)$. Using the locality of weak gradients in the base spaces, this immediately
yields $J(f_n)\to J(f)$ and completes the proof of the inequality  $\C_c\leq J$ on $H^{1,2}(Z,d,\mm)$.

Now, let us prove the more delicate inclusion $D(J)\subset H^{1,2}(Z,d,\mm)$. In this proof it will be useful to assume that the distances $d_X$ and
$d_Y$ are geodesic; this is not restrictive, since the geodesic distances associated to $d_X$ and $d_Y$
are (because of doubling and Poincar\'e) equivalent to the original distances and even induce the same weak gradients
(see for instance \cite[Corollary~7.3.17]{HKST}, but note that we need only equivalence of distances). We will prove that
any $h\in D(J)$ whose support is contained in a product of balls $\overline{B}_R(x_0)\times \overline{B}_R(y_0)$ belongs to 
$H^{1,2}(Z,d,\mm)$. The
general case can be easily achieved by approximation. 

Let us proceed now to the proof of the inclusion $D(J)\subset H^{1,2}(Z,d,\mm)$.
Notice that, by the very definition of $J$,
for any $h\in D(J)$ one has not only $h^x\in H^{1,2}(Y,d_Y,\mm_Y)$ for $\mm_X$-a.e. $x\in X$ and $h^y\in H^{1,2}(X,d_X,\mm_X)$ for $\mm_Y$-a.e. $y\in Y$,
but also $\C_Y(h^x)\in L^1(X,\mm_X)$ and $\C_X(h^y)\in L^1(Y,\mm_Y)$. Thanks to Cheeger's theory
(see also Theorem~40 of the recent paper \cite{AmCoDi} for a different construction based on difference quotients and 
$\Gamma$-convergence which uses only the doubling property of the metric space)
we can find $L^2$-lower semicontinuous quadratic forms ${\cal E}_X$ in $L^2(X,\mm_X)$ and ${\cal E}_Y$ in $L^2(Y,\mm_Y)$
equivalent to the Cheeger energies, namely
\begin{equation}\label{eq:equivchee}
c_1\,\C_X\leq {\cal E}_X\leq \frac 1{c_1} \C_X,\qquad
c_2\,\C_Y\leq {\cal E}_Y\leq \frac 1{c_2} \C_Y
\end{equation}
for suitable positive structural constants $c_1,\,c_2$. The quadratic forms
$$
{\cal E}_{X,1}(f):={\cal E}_X(f)+\int_X f^2\,d\mm_X,\qquad
{\cal E}_{Y,1}(g):={\cal E}_Y(g)+\int_Y g^2\,d\mm_Y.
$$
induce Hilbertian structures in $H^{1,2}(X,d_X,\mm_X)$ and $H^{1,2}(Y,d_Y,\mm_Y)$ respectively. We also denote
by $J_1$ the quadratic form
$$
J_1(h):= \int_X {\cal E}_{Y,1}(f^x)\,d\mm_X(x)+\int_Y{\cal E}_{X,1}(f^y)\,d\mm_Y(y)
$$
and notice that $D(J_1)=D(J)$, because of \eqref{eq:equivchee}. 

Fix $x_{0} \in X,\, y_{0} \in Y$ and $R>0$. Fix the notation:
\[D(x_{0},R)=\{f\colon X\to \mathbb{R} \colon {\mathrm{supp}}(f) \subset \overline{B}(x_{0},R)\},\]
\[D(y_{0},R)=\{g\colon Y\to \mathbb{R} \colon {\mathrm{supp}}(g) \subset \overline{B}(y_{0},R)\}.\]

Then, for every $M>0$, the sets
\[\{f\in D(x_{0},R) \colon {\cal E}_{X,1}(f)\leq M\},\] 
\[\{g\in D(y_{0},R) \colon {\cal E}_{Y,1}(g)\leq M\},\]
are compact in $L^2(X,\mm_X)$ and $L^2(Y,\mm_Y)$ respectively. Indeed, \cite[Theorem~8.1]{HK} gives this Rellich-Kondrachov type result for domains in which the measure is doubling and also a global Poincar\'{e} inequality result holds (see \cite[Inequality (46)]{HK}). However, it is easy to see that balls in geodesic spaces are John domains (see Section~9.1 in \cite{HK} for the definition) and hence \cite[Theorem~9.7]{HK} provides the required global Poincar\'{e} inequality. Let
\[H_{X}=L^2(X,\mm_X)\cap D(x_{0},R),\]
\[H_{Y}=L^2(Y,\mm_Y)\cap D(y_{0},R).\] 
Notice that $H^{1,2}(X,d_{X},\mm_{X})\cap D(x_{0},R)$ is dense in $H_{X}$ with respect to the $L^{2}(X,m_{X})$ norm and 
similarly for $H^{1,2}(Y,d_{Y},\mm_{Y})\cap D(y_{0},R)$. Hence, by applying 
Theorem~\ref{thm:Bassano} to ${\cal E}_{X,1}$ in
\[H_{X}=L^2(X,\mm_X)\cap D(x_{0},R)\] 
and to ${\cal E}_{Y,1}$ in
\[H_{Y}=L^2(Y,\mm_Y)\cap D(y_{0},R),\] 
we can find complete orthonormal bases $(f_n)$ of $H_{X}$ and $(g_n)$ of $H_{Y}$
made of eigenvectors of the operators $L_X$ and $L_Y$ associated to these quadratic forms. In particular:
\begin{itemize}
\item[(a)] the linear semigroup $P^X_t$ associated to 
${\cal E}_{X,1}$ on $H_{X}$ leaves all $1$-dimensional vector spaces $\R f_n$ invariant, and an analogous property holds for $P^Y_t$;
\item[(b)] the functions $f_n$ are mutually orthogonal for the scalar product induced by ${\cal E}_{X,1}$ and an analogous 
property holds for $g_m$.
\end{itemize}  
Now, it is easily seen that the vector space $E$ spanned by the tensor products $f_ng_m$, where $f_{n}\in H^{1,2}(X,d_{X},m_{X})\cap D(x_{0},R)$ and $g_{m}\in H^{1,2}(Y,d_{Y},m_{Y})\cap D(y_{0},R)$, is dense in the space
\[\{h\in D(J)\colon {\mathrm{supp}}(h)\subset \overline{B}(x_{0},R)\times \overline{B}(y_{0},R)\}\]
w.r.t. the norm induced by $J_1$.
Indeed, a simple functional analytic argument 
is based on the fact that the linear semigroup $P_t$ induced by the quadratic form $J_1$ leaves $E$ invariant, thanks to (a) and
$$
P_t (f_ng_m)=(P^X_t f_n)(P^Y_t g_m)\qquad t\geq 0,
$$
and on the fact that $E$ is dense in $L^2(Z,\mm)$ norm; these two facts (see for instance \cite[Proposition~4.9]{AGS11b})
imply density in the stronger norm induced by $J_1$.

In order to conclude the proof we use first the convexity and $1$-homogeneity of $\sqrt{\C}$ and then
Lemma~\ref{lem2} to get
\begin{eqnarray*}
\sqrt{\C}(h)&\leq&\sum_{n,m}|\lambda_{n,m}|\sqrt{\C}(f_ng_m)=
\sum_{n,m}|\lambda_{n,m}|\sqrt{J}(f_ng_m)\\&\leq&
c\sum_{n,m}|\lambda_{n,m}|\sqrt{J_1}(f_ng_m)\leq c\biggl(\sum_{n,m} \lambda_{n,m}^2 J_1(f_n g_m)\biggr)^{1/2}\\
&=& c\sqrt{J_1}(h)
\end{eqnarray*}
for any $h\in E$ written as a sum $\sum_{n,m} \lambda_{n,m}f_ng_m$, with only finitely many
$\lambda_{n,m}$ nonzero, where in the last equality we used the fact that $(f_ng_m)$ are mutually orthogonal 
w.r.t. the scalar product induced by $J_1$ (as a consequence of (b)). Finally we can use the density of $E$ and the lower semicontinuity of $\C$ to conclude
that $\C\leq c^{2}J_1$ on the set
\[\{h\in D(J)\colon {\mathrm{supp}}(h)\subset \overline{B}(x_{0},R)\times \overline{B}(y_{0},R)\}.\]
By letting $R\to \infty$ we deduce $\C$ is finite on $D(J)$ so that $D(J)\subset D(\C)$.
\end{proof}

A simple consequence of \eqref{eq:partial_lap2} is that, independently of quadraticity assumptions on
$\C_X$ and $\C_Y$, the semigroup $P_t$ associated to $J$ (namely
the gradient flow of $J$) acts on tensor products $f_Xf_Y$ as follows
\begin{equation}\label{eq:partial_lap1}
P_t(f_Xf_Y)=(P_t^X f_X)(P_t^Y f_Y),
\end{equation}
where $P_t^X$ and $P_t^Y$ are the semigroups on the factors. Indeed, \eqref{eq:partial_lap2} with $\xi_X=-\Delta P^X_t f_X$ 
and $\xi_Y=-\Delta_YP_t^Yf_Y$ and the Leibniz rule give that
$$
-\frac{d}{dt} (P_t^X f_X)(P_t^Y f_Y)\in \partial J\bigl((P_t^X f_X)(P_t^Y f_Y)\bigr),
$$
which is the subdifferential formulation of the gradient flow.

\begin{proof} [of Theorem~\ref{thm:Cquad}]
Assuming that $\C_X$, $\C_Y$ are quadratic, it is clear that \eqref{eq:factor} implies that $\C$ are quadratic, so let us
prove the converse implication from $\C$ quadratic to \eqref{eq:factor}. By Lemma~\ref{lem1} and the equality
$\C_c=\C$ provided by Theorem~\ref{thm:I=C}, we know that the two 
lower semicontinuous quadratic forms
$$
Q_1(f):=\C(f),\qquad Q_2(f):=J(f)
$$
satisfy $Q_1\geq Q_2$, hence $Q=Q_1-Q_2$ is a nonnegative quadratic form which, by Lemma~\ref{lem2},
vanishes on tensor products. By the inequality $Q(u+v)\leq 2Q(u)+2Q(v)$, valid for nonnegative quadratic forms,
we obtain that $Q$ vanishes on the vector space spanned by tensor products. Now, as in the proof of Theorem~\ref{thm:PI_tensor}, 
it is easily seen that the vector space $E$ spanned by the tensor products is dense w.r.t. the norm $\|\cdot\|_{Q_2}:=\sqrt{\|\cdot\|_2^2+Q_2^2}$.
Indeed, the linear semigroup $P_t$ induced by the quadratic form $Q_2$ leaves $E$ invariant (because of
\eqref{eq:partial_lap1}) and
$E$ is dense in $L^2(Z,\mm)$ norm; these two facts (see for instance \cite[Proposition~4.9]{AGS11b})
imply density in the stronger norm induced by $\|\cdot\|_{Q_2}$. Therefore for any
$f\in L^2(Z,\mm)$ we can find $f_n\in E$ convergent to $f$ in $L^2(Z,\mm)$ with $Q_2(f_n)\to Q_2(f)$. Using lower
semicontinuity of $Q_1$ we get
$$
Q_1(f)\leq\liminf_{n\to\infty}Q_1(f_n)=\liminf_{n\to\infty}Q_2(f_n)=Q_2(f).
$$
Taking into account the inequality $|\nabla f^x|_{*,Y}^2+|\nabla f^y|_{*,X}^2\leq |\nabla f|_c^2(x,y)$ for $\mm$-a.e. 
$(x,y)\in Z$ provided by the combination of Theorem~\ref{thm:I=C} and Lemma~\ref{lem1}, we obtain \eqref{eq:factor}.

To prove the last statement of the theorem, assume that the factors are strongly asymptotically Hilbertian. Notice that 
$$
f\mapsto\int_Z |\nabla f|_c^2\,d\mm=\int_X \biggl(\int_Y|\nabla f^x|^2\,d\mm_Y\biggr)\,d\mm_X(x)+
\int_Y\biggl(\int_X|\nabla f^y|^2\,d\mm_X\biggr)\,d\mm_Y(y)
$$
is a quadratic form on ${\rm Lip}_{\rm loc}(Z)$. Since the lower semicontinuous relaxation of a quadratic form is still a quadratic
form, it follows that $\C=\C_c$ is a quadratic form. Therefore we can apply the first part of the statement
to obtain \eqref{eq:factor}.
\end{proof}

\section{The Sobolev space $H^{1,1}(X,d,\mm)$}\label{sec:4}

In this section we investigate in more detail the possible definitions of the space $H^{1,1}$ described in
\cite{AmDi}. We adopt the definition leading to the larger space and study a few structural properties. Then,
assuming doubling and $(1,1)$-Poincar\`e, we compare with more restrictive definitions.
The results of this section will be applied in the next section, dealing with $H^{1,1}$ functions in
$X\times\R$. In this section we assume that $\mm$ is finite on bounded sets.

\begin{definition} [The space $H^{1,1}(X,d,\mm)$] \label{def:H11} We denote by $H^{1,1}(X,d,\mm)$ the subspace of all $f\in BV(X,d,\mm)$
satisfying:
\begin{itemize}
\item[(i)] $f\circ\gamma\in H^{1,1}(0,1)$ for $1$-almost every $\gamma\in C([0,1];X)$;
\item[(ii)] $|Df|\ll\mm$.
\end{itemize} 
By analogy with the Sobolev case, we denote by $|\nabla f|_{*,1}$ the density of $|Df|$ w.r.t. $\mm$.
\end{definition}

In view of the example mentioned in Remark~\ref{rem:bad_things_occur}, we know that (ii) would not be sufficient to provide
Sobolev regularity of $f$ along $1$-almost every curve (not even if $|Df|$ has a bounded density), 
a natural requirement for Sobolev functions,
also in view of the connection with other definitions. In the presence of doubling and $(1,1)$-Poincar\'e inequality for locally Lipschitz
functions, however, (ii) is sufficient, see the proof of the implication from (iii) to (i) in Theorem~\ref{thm:equivSob1} which only
uses $|Df|\ll\mm$ (see also \cite[Theorem~4.6]{HKLL14}).

We use the notation $|\nabla f|_{*,1}$ because, at this level of generality, we expect that weak gradients depend on the
integrability exponent, even for Lipschitz functions, see \cite{DiSp} for examples compatible even with the doubling
assumption.
Notice that the obvious inequality $|Df|\leq |\nabla f|\mm$ on locally Lipschitz functions gives $|\nabla f|_{*,1}\leq |\nabla f|$
$\mm$-a.e. in $X$ for all $f\in {\rm Lip}_{\rm loc}(X)$, while additional assumptions are needed to reverse this inequality (see
\eqref{eq:atleast} below).

In the next proposition we are able to prove the locality of $|\nabla f|_{*,1}$. 
It was difficult for us to prove this fact using only \eqref{optimal_sequence_BV} (because, unlike the Sobolev case, the convergence of
slopes is weak and not strong), therefore we adopt the dual point of
view provided by Theorem~\ref{thm:amdm}. 

\begin{proposition}[Locality of $|\nabla f|_{*,1}$]  \label{plocaBV} Let $f\in H^{1,1}(X,d,\mm)$. Then
\begin{equation}\label{eq:plocaBV}
|\nabla f|_{*,1}=0\qquad\text{$\mm$-a.e. in $\{f=0\}$.}
\end{equation}
In particular, if $f,\,g$ belong to $H^{1,1}(X,d,\mm)$, then
\begin{equation}\label{eq:plocaBV1}
|\nabla f|_{*,1}=|\nabla g|_{*,1}\qquad\text{$\mm$-a.e. in $\{f=g\}$.}
\end{equation}
\end{proposition}
\begin{proof} It is well known that for a Sobolev function $g:(0,1)\to\R$ the property holds, namely
$g'=0$ $\Leb{1}$-a.e. on $\{g=0\}$. Given a $\infty$-test plan $\ppi$, from part (i) of Definition~\ref{def:H11} it follows that
$$
(f\circ\gamma)'=0\quad\text{$\Leb{1}$-a.e. on $\gamma^{-1}(\{f=0\})$,
for $\ppi$-a.e. $\gamma$.}
$$
Since $|D(f\circ\gamma)|=|(f\circ\gamma)'|\Leb{1}$,
the definition of push-forward gives $\gamma_\sharp |D(f\circ\gamma)|(\{f=0\})=0$
for $\ppi$-a.e. $\gamma$. By integration w.r.t. $\ppi$, since $\ppi$ is arbitrary we get that the measure
$$
\mu:=\chi_{\{f\neq 0\}}|\nabla f|_{*,1}\mm=\chi_{\{f\neq 0\}}|Df|
$$
still satisfies (ii) of Theorem~\ref{thm:amdm}. Then, the minimality of $|Df|=|\nabla f|_{*,1}\mm$ gives \eqref{eq:plocaBV}.

It follows from the inequalities (derived from \eqref{eq:subadd}) that 
$$
-|\nabla (f-g)|_{*,1}+|\nabla g|_{*,1} \leq|\nabla f|_{*,1}\leq |\nabla g|_{*,1}+|\nabla (f-g)|_{*,1}
\qquad\text{$\mm$-a.e. in $X$.}$$
Using \eqref{eq:plocaBV} we obtain \eqref{eq:plocaBV1}
\end{proof}

If we assume that the $(1,1)$-Poincar\'e inequality for Lipschitz functions
we can specialize \eqref{eq:PIBV} to $H^{1,1}(X,d,\mm)$, obtaining
\begin{equation}\label{eq:PIH11}
\media_{B_r(x)}| f-f_{x,r}|\,d\mm\leq c r \int_{B_{\lambda r}(x)}|\nabla f|_{*,1}\,d\mm\qquad\forall f\in H^{1,1}(X,d,\mm).
\end{equation}
Analogously, \eqref{eq:liploc} gives
\begin{equation}\label{eq:liploc2}
|\tilde f(x)-\tilde f(y)|\leq c d(x,y)\bigl( M_r|\nabla f|_{*,1}(x)+M_r|\nabla f|_{*,1}(y)\bigr). 
\end{equation}
When $f$ is locally Lipschitz we can use \eqref{eq:liploc2} at approximate continuity points of $M_r|\nabla f_{*,1}|$ 
(see for instance the argument in \cite{AmCoDi}) to get
$$
|\nabla f|\leq cM_r(|\nabla f|_{*,1})\qquad\text{$\mm$-a.e. in $X$, for all $f\in {\rm Lip}_{\rm loc}(X)$.}
$$
Letting $r\downarrow 0$ eventually gives, for some structural constant $c$,
\begin{equation}\label{eq:atleast}
|\nabla f|\leq c|\nabla f|_{*,1}\qquad\text{$\mm$-a.e. in $X$, for all $f\in {\rm Lip}_{\rm loc}(X)$.}
\end{equation}

Now, we can combine the ``local'' Lipschitz estimate \eqref{eq:liploc2}, 
with the weak $L^1$ and asymptotic estimates \eqref{eq:maximal_estimates1},
\eqref{eq:maximal_estimates2} to get
the following Lusin type approximation result of $H^{1,1}$ functions by Lipschitz functions; this way, we get the equivalence
with other, a priori stronger, definitions.

\begin{theorem} [Equivalent definitions of $H^{1,1}(X,d,\mm)$] \label{thm:equivSob1}
Let $f\in L^1(X,\mm)$ and let us consider the following conditions:
\begin{itemize}
\item[(i)] there exist $f_n\in {\rm Lip}(X)\cap L^1(X,\mm)$ with $f_n\to f$ in $L^1(X,\mm)$, $\mm(\{f\neq f_n\})\to 0$ and
$|\nabla f_n|$ equi-integrable;
\item[(ii)] there exist a Borel function $g:X\to [0,\infty)$ and a representative $\hat{f}$ of $f$ such that 
$\int_X g\,d\mm<\infty$ and
\begin{equation}\label{eq:treponzio1}
|\hat{f}(\gamma(1))-\hat{f}(\gamma(0))|\leq\int_\gamma g
\qquad\text{for ${\rm Mod}_1$-a.e. curve $\gamma$;}
\end{equation}
\item[(iii)] $f\in H^{1,1}(X,d,\mm)$.
\end{itemize} 
Then (i)$\Rightarrow$(ii), with $g$ limit point of $|\nabla f_n|$ in the weak $L^1(X,\mm)$ topology, and (ii)$\Rightarrow$(iii), with
\begin{equation}\label{eq:treponzio2}
|\nabla f|_{*,1}\leq g\qquad\text{$\mm$-a.e. in $X$.}
\end{equation}
If $(X,d,\mm)$ is doubling and the $(1,1)$-Poincar\'e inequality for locally Lipschitz functions holds, then
$f\in BV(X,d,\mm)$ and $|Df|\ll\mm$ imply (i). In particular (iii)$\Rightarrow$(i).
\end{theorem}
\begin{proof} (i)$\Rightarrow$(ii). 
See for instance \cite{Sh}, \cite[Theorem~10.5]{Hei} and the recent monograph \cite{HKST} for the construction of the representative $\hat{f}$. The main point
is to find Lipschitz functions $\tilde{f}_n$ convergent to $f$ $\mm$-a.e. in $X$ and upper gradients $g_n$ of $f$
strongly convergent to $g$ in $L^1(X,\mm)$, using also the fact that this implies $\int_\gamma g_n\to\int_\gamma g$
for ${\rm Mod}_1$-a.e. $\gamma$. The functions $\tilde{f}_n$ are finite convex combinations of $f_n$, where the
coefficients of the convex combinations are chosen in such a way that the corresponding convex combination $g_n$
of $|\nabla f_n|$, which are upper gradients of $\tilde{f}_n$, are strongly convergent.

\noindent 
(ii)$\Rightarrow$(iii). 
Since families of subcurves of a ${\rm Mod}_1$-negligible set of curves are still ${\rm Mod}_1$-negligible,
the argument in \cite{Sh} shows that $\hat{f}$ is absolutely continuous (and not only $H^{1,1}$) along ${\rm Mod}_1$-a.e.
curve, with
\begin{equation}\label{eq:trani1}
\bigl|\frac{d}{dt}f\circ\gamma(t)\bigr|\leq g(\gamma(t))|\dot\gamma(t)|\qquad\text{for $\Leb{1}$-a.e. $t\in (0,1)$, for
${\rm Mod}_1$-a.e. $\gamma$.}
\end{equation}
Since we already observed that ${\rm Mod}_1$-negligible sets are $1$-negligible, we obtain (i) of Definition~\ref{def:H11}.
In connection with condition (ii) of Definition~\ref{def:H11}, we can use \eqref{eq:trani1} to show that
the measure $\mu=g\mm$ satisfies 
$$
\int \gamma_\sharp|D(f\circ\gamma)|\,d\ppi(\gamma)\leq C(\ppi)\|{\rm Lip}(\gamma)\|_\infty \mu
$$
for any $\infty$-test plan $\ppi$. The minimality property of $|Df|$ stated in Theorem~\ref{thm:amdm}
then gives $|Df|\leq g\mm$, so that $|Df|\ll\mm$.

\noindent
(iii)$\Rightarrow$(i), under the doubling and Poincar\'e assumptions.
We apply \eqref{eq:liploc} with $r=\infty$ and denote by $E_n$ the set of approximate continuity points of $f$
where $M_\infty(|\nabla f|_{*,1}+|f|)$ is smaller than $n$. By McShane Lipschitz extension theorem, we can extend 
$\tilde f\vert_{E_n}$ to a $2cn$-Lipschitz function on $X$, denoted by $f_n$, with $|f_n|\leq n$. 
Since the weak $L^1$ estimate holds, it is then clear that $f_n\in L^1(X,\mm)$ and 
$\mm(\{f\neq f_n\})\leq\mm(X\setminus E_n)\to 0$.
In connection with equi-integrability of $|\nabla f_n|$, namely
$$
\lim_{z\to\infty}\limsup_{n\to\infty}\int_{\{|\nabla f_n|>z\}}|\nabla f_n|\,d\mm=0,
$$
it suffices to split the integral on $E_n$ and on $X\setminus E_n$. The former can be estimated uniformly in $n$ from above, thanks to 
\eqref{eq:atleast} and locality, with
$$
c\int_{E_n\cap\{|\nabla f_n|_{*,1}>z/c\}}|\nabla f_n|_{*,1}\,d\mm\leq
c\int_{\{|\nabla f|_{*,1}>z/c\}}|\nabla f|_{*,1}\,d\mm
$$
which is infinitesimal as $z\to\infty$.
The latter can be estimated using \eqref{eq:maximal_estimates2}. A similar and simpler argument shows
that $f_n\to f$ in $L^1(X,\mm)$.
\end{proof}

The characterization (ii) in Theorem~\ref{thm:equivSob1} of $H^{1,1}(X,d,\mm)$ suggests another definition of minimal
$1$-gradient $|\nabla f|_{w,1}$, namely the smallest function $g$ (up to $\mm$-negligible sets) 
such that, for some representative $\hat{f}$ of $f$, the inequality
\eqref{eq:treponzio1} holds. This is the point of view also adopted, besides the relaxation point of view,
in the theory of Sobolev spaces with exponent $p>1$, see \cite{Sh}, \cite{HKST} and \cite{AGS12} for a comparison between
the definitions. From \eqref{eq:treponzio2} we immediately get
\begin{equation}\label{eq:treponzio3}
|\nabla f|_{*,1}\leq |\nabla f|_{w,1}\qquad\text{$\mm$-a.e. in $X$,}
\end{equation}
while the inequality \eqref{eq:atleast} gives, for some structural constant $c$,
\begin{equation}\label{eq:treponzio4}
|\nabla f|_{w,1}\leq |\nabla f|\leq c|\nabla f|_{*,1}\qquad\text{$\mm$-a.e. in $X$.}
\end{equation}

\begin{remark}\label{rem:do_not_coincide}{\rm 
In general equality does not hold in \eqref{eq:treponzio3}, not even under the doubling and Poincar\'e
assumptions: in \cite{HKLL14}, a metric measure space is built
endowing $X=[0,1]$ with the Euclidean distance and the weighted measure $\mm=\omega\Leb{1}\in\Probabilities{X}$, with $\omega\equiv 1$
on a ``fat'' Cantor set $K\subset (0,1)$ (i.e. a compact totally disconnected set with positive Lebesgue measure) and $\omega=1/2$, say, on
$(0,1)\setminus K$. It is clear that $(X,d,\mm)$, being comparable to the standard Euclidean structure, satisfies the
doubling and Poincar\'e assumptions. On the other hand, in \cite{HKLL14} Lipschitz functions $f_n$ convergent to the 
identity function $f$ in $L^1(X,\mm)$ are built in such a way that
$$
\limsup_{n\to\infty}\int_X|\nabla f_n|\,d\mm<\int_X |\nabla f|\,d\mm.
$$
Now, the very definition of $H^{1,1}(X,d,\mm)$ gives that $|Df|(X)<\int_X|\nabla f|\,d\mm$,
hence $|Df|=|\nabla f|_{*,1}\mm$ gives that $|\nabla f|_{*,1}<|\nabla f|$ in a set with positive $\mm$-measure. Since it is easy to check
that $|\nabla f|=|\nabla f|_{w,1}=1$ $\mm$-a.e. for the identity function, we have a situation where
the two notions of weak gradient in the limiting case $p=1$ differ. The difference of the notions arises
basically from the fact that $f\mapsto\int_X |\nabla f|_{*,1}\,d\mm$ has a $L^1$-lower semicontinuity property built in, while
$f\mapsto\int_X|\nabla f|_{w,1}\, d\mm$ a priori does not. \fr
}\end{remark}

In the following proposition we show that $|\nabla f|_{w,1}$ enjoys a stronger approximation property
by Lipschitz functions, compared to \eqref{optimal_sequence_BV}. In the proof we will use the identity
\begin{equation}\label{eq:cdc}
|\nabla f|=|\nabla f|_{w,1}\qquad\text{$\mm$-a.e. in $X$, for all $f\in {\rm Lip}_{\rm loc}(X)$,}
\end{equation}
which is proved in \cite[Theorem~12.5.1]{HKST} and it extends Theorem~\ref{thm:cheeger} to the limiting case $p=1$. 
For the sake of completeness we sketch the main ideas used to prove \eqref{eq:cdc}. Clearly, it suffices to prove that for every $f\in {\rm Lip}_{\rm loc}(X)$ 
it holds $|\nabla f|\leq |\nabla f|_{w,1}$ $\mm$-a.e. in $X$. By the Vitali-Carath\'eodory theorem (see \cite[Theorem~3.2.4]{HKST}),
it suffices to show that $|\nabla f|\leq g$ for any countably valued lower semicontinuous
upper gradient $g\in L^{\infty}_{\rm loc}(X)$ of $f$ for which there is a positive real number $c$ with $g\geq c$. By \cite[Lemma 12.5.9]{HKST} there is sequence $(f_k)$ of Lipschitz functions, with a
corresponding sequence $(g_k)$ of continuous upper gradients, such that $(f_k)$ converges to $f$ in
$L^2_{\rm loc}(X)$ and $\limsup_k g_k\leq g$ $\mm$-a.e. in $X$. Then $(f_k)$ is bounded in $H^{1,2}_{\rm loc}(X)$ and we conclude exactly as in the proof of \cite[Theorem~12.5.1]{HKST}. 

\begin{proposition}\label{prop:strongapp} Assume that $(X,d,\mm)$ is doubling and that the $(1,1)$-Poincar\'e inequality
holds for locally Lipschitz functions. Then for all $f\in H^{1,1}(X,d,\mm)$ there exist locally Lipschitz functions $f_n$ with
$f_n\to f$ and $|\nabla f_n| \to |\nabla f|_{w,1}$ in $L^1(X,\mm)$.
\end{proposition}
\begin{proof} Using the identity \eqref{eq:cdc} it suffices to prove $|\nabla f_n|_{w,1}\to |\nabla f|_{w,1}$. Now, notice
that, because of the inequality \eqref{eq:treponzio4}, the weak gradient $|\nabla g|_{w,1}$ is local on Borel sets as well
as $|\nabla g|_{*,1}$ (but, this property could be proved directly with the definition of $|\nabla g|_{w,1}$). Therefore the
sequence $(f_n)$ provided by Theorem~\ref{thm:equivSob1}(i) provides the result.
\end{proof}

\section{Sets of finite perimeter and the area formula}\label{sec:per}

In the setup of the tensorization Section~\ref{sec:tenso}, assume that $Y=\R$ endowed with the standard Euclidean structure
and the Lebesgue measure $\Leb{1}$.
This is a $\sigma$-finite and not finite metric measure space, but we can use a localization argument for weak gradients
(in the same spirit of \cite[Theorem~1.1]{AmDi} or \cite[Lemma~4.11]{AGS11a}) and apply all
results of that section to this situation. Departing a bit from the notation of Section~\ref{sec:tenso}, we  
will then consider a m.m. space $(X,d,\mm)$ and endow $X\times\R$ with the product distance $\tilde d$ and
the product measure $\tilde\mm=\mm\times\Leb{1}$. 

If we consider a set of finite perimeter $E\subset X\times \R$ that is the subgraph of a function $f:X\to\R$, namely
$$
E_f:=\bigl\{(x,t)\in X\times\R:\ t<f(x)\bigr\}
$$
then it is natural to compare the perimeter of $E_f$, defined according to the well established metric $BV$ theory in $X\times\R$, to the
area of the graph of $f$. In this context the natural regularity condition on $f$ is $f\in H^{1,1}(X,d,\mm)$ (or even
$BV(X,d,\mm)$); namely we would like to know whether
\begin{equation}\label{eq:goodper}
P(E_f,B\times\R)=\int_B\sqrt{1+|\nabla f|_{*,1}^2}\,d\mm\qquad\text{for all $B\subset X$ Borel}
\end{equation}
for $f\in H^{1,1}(X,d,\mm)$.
We will provide a partial positive answer using the tensorization property of weak gradients; this is 
not surprising, since locally Lipschitz functions on $A$ are used to define the perimeter $P(E_f,A)$ for $A\subset X\times\R$
open.

\begin{theorem} Assume that $(X,d,\mm)$ is a doubling metric measure space, and that the $(1,1)$-Poincar\'e inequality holds
for Lipschitz functions. Then:
\begin{itemize}
\item[(a)] for all $f\in H^{1,1}(X,d,\mm)$ and all $B\subset X$ Borel one has
\begin{equation}\label{eq:goodper11}
P(E_f,B\times\R)\leq\int_B\sqrt{1+|\nabla f|_{w,1}^2}\,d\mm.
\end{equation}
\item[(b)] for all $f\in BV(X,d,\mm)$, denoting by $|Df|=|\nabla f|_{*,1}\mm+|D^s f|$ the Radon-Nikodym decomposition of $|Df|$,
one has
\begin{equation}\label{eq:goodper12}
P(E_f,B\times\R)\geq \int_B\sqrt{1+|\nabla f|_{*,1}^2}\,d\mm+|D^s f|(B)\qquad\text{for all $B\subset X$ Borel.}
\end{equation}
\end{itemize}
\end{theorem}
\begin{proof} We first prove the inequality $\leq$ in \eqref{eq:goodper11}, for $f$ locally Lipschitz and $B\subset X$ open. 
Let $\chi_\eps\in C^\infty(\R)$
be monotonically convergent to $\chi_{(0,\infty)}$ with $0\leq\chi_\eps\leq 1$, $\chi_\eps'\geq 0$ and $\int\chi_\eps'\,dt\leq 1$. Then,
$$
g_\eps(x,t):=\chi_\eps(f(x)-t)
$$
provides a locally Lipschitz approximation of the characteristic function of $E_f$. Given our assumptions on $(X,d,\mm)$, Theorem~\ref{thm:cheeger} (applied in $(X\times\R,\tilde d,\tilde \mm)$ which is doubling and satisfies $(1,1)$-Poincar\'e for locally Lipschitz functions), Theorem~\ref{thm:PI_tensor} and the chain rule for weak gradients give
\begin{eqnarray*}
|\nabla g_\eps|(x,t)&=&|\nabla g_\eps|_{*}(x,t)=\sqrt{|\nabla g_\eps^x|^2_{*,\R}(t)+|\nabla g_\eps^t|^2_{*,X}(x)}\\&=&
\chi_\eps'(f(x)-t)\sqrt{1+|\nabla f|_*^2(x)}=\chi_\eps'(f(x)-t)\sqrt{1+|\nabla f|^2(x)}
\end{eqnarray*}
for $\tilde\mm$-a.e. $(x,t)\in X\times\R$.
If we integrate on $B\times\R$ and use Fubini's theorem, we get
$$
\int_{B\times\R}|\nabla g_\eps|\,d\tilde\mm\leq\int_B\sqrt{1+|\nabla f|^2}\,d\mm.
$$
By the definition of $P(E_f,B\times\R)$ we obtain the inequality 
\begin{equation}\label{eq:june18a}
P(E_f,B\times\R)\leq\int_B\sqrt{1+|\nabla f|^2}\,d\mm.
\end{equation}

For $f\in H^{1,1}(X,d,\mm)$ and $B$ open we use the Lipschitz approximation provided by Proposition~\ref{prop:strongapp} and the lower semicontinuity
of the perimeter in open sets to obtain \eqref{eq:goodper11}. Being an inequality between positive and finite Borel measures, 
it extends from open to Borel sets.

We now prove the inequality \eqref{eq:goodper12}. Let $g_n\in {\rm Lip}_{{\rm loc}}(X\times (-N,N))$ be convergent to $\chi_{E_f}$
in $L^1(X\times (-N,N),\tilde\mm)$. By Fubini's theorem, possibly refining the sequence, we can assume that
$g_n^t\to\chi_{E_f}^t=\chi_{\{f(x)>t\}}$ in $L^1(X,\mm)$ for $\Leb{1}$-a.e. $t\in (-N,N)$, and that
$g_n^x\to\chi_{E_f}^x=\chi_{\{t<f(x)\}}$ in $L^1(-N,N)$ for $\mm$-a.e. $x\in X$. 
Applying once more Theorem~\ref{thm:cheeger} in $X\times\R$ and Theorem~\ref{thm:Cquad} gives
$|\nabla g_n|^2(x,t)=|\nabla g_n^x|(t)^2+|\nabla g_n^t|^2(x)$ $\tilde\mm$-a.e. in $X\times (-N,N)$.
Fix now $a,\,b:X\to [0,1]$ continuous with $a^2+b^2\leq 1$ and notice that
\begin{eqnarray*}
&&\liminf_{n\to\infty}
\int_{X\times (-N,N)}\sqrt{|\nabla g_n^x|^2(t)+|\nabla g_n^t|^2(x)}\,d\tilde\mm(x,t)\\&\geq&
\liminf_{n\to\infty}
\int_{X\times (-N,N)}a(x)|\nabla g_n^x|(t)+b(x)|\nabla g_n^t|(x)\,d\tilde\mm(x,t)\\
&\geq&
\int_X\liminf_{n\to\infty}
\int_{(-N,N)}a(x)|\nabla g_n^x|(t)\,dt d\mm(x)+
\int_{(-N,N)}\liminf_{n\to\infty}
\int_Xb(x)|\nabla g_n^t|(x)\, d\mm(x)dt\\&\geq&
\int_X a\chi_{f\in (-N,N)}\,d\mm+\int_{(-N,N)}\int_Xb \,d|D\chi_{\{f>t\}}|\,dt.
\end{eqnarray*}
Now we use the fact that $(g_n)$ is arbitrary to get
$$
P(E_f,X\times (-N,N))\geq \int_X a\chi_{f\in (-N,N)}\,d\mm+\int_{(-N,N)}\int_Xb \,d|D\chi_{\{f>t\}}|\,dt.
$$
Letting $N\uparrow\infty$ and using the coarea formula 
$$
|Df|=\int_{-\infty}^\infty |D\chi_{\{f>t\}}|\, dt
$$
(see for instance \cite{Mi}) gives
$$
P(E_f,X\times\R)\geq \int_X a\,d\mm+\int b\,d|Df|=
\int_X (a+b|\nabla f|_{*,1})\,d\mm+\int_X b\,d|D^sf|.
$$
Since $|D^sf|\perp\mm$, taking the supremum among all admissible pairs $(a,b)$ we get
\eqref{eq:goodper12} with $B=X$.
Repeating the argument
with any open set $B\subset X$ the inequality is proved on all open sets and then on all Borel sets.
\end{proof}

In the proof of the statement made in the next remark, the following lemma will be useful.

\begin{lemma}\label{lem:strictconvexity}
Let $u_n\in L^1(X,\mm)$ be nonnegative functions satisfying
$$
\limsup_{n\to\infty}\int_X \sqrt{1+u_n^2}\,d\mm\leq\int_X\sqrt{1+u^2}\,d\mm
$$
for some nonegative $u\in L^1(X,\mm)$. If $u_n\mm$ weakly converge to a measure 
$\mu\geq u\mm$, then $\mu=u\mm$ and $u_n\to u$ in $L^1(X,\mm)$.
\end{lemma}
\begin{proof} Let $\nu=\frac{\nu}{\mm}\mm+\mu^s$ be the Radon-Nikodym decomposition of 
$\nu$ with respect $\mm$. Since the functional 
$$
\nu\mapsto \int_X\sqrt{1+\bigl|\frac{\nu}{\mm}\bigr|^2}\,d\mm+\nu^s(X)
$$
is sequentially weakly lower semicontinuous and the density of $\mu$ is, by assumption, larger than $u$ $\mm$-a.e.,
we obtain
$$
\liminf_{n\to\infty}\int_X \sqrt{1+u_n^2}\,d\mm\geq
\int_X \sqrt{1+\bigl|\frac{\mu}{\mm}\bigr|^2}\,d\mm+\mu^s(X)\geq
\int_X\sqrt{1+u^2}\,d\mm+\mu^s(X),
$$
hence $\mu^s(X)=0$, $\mu=u\mm$ and $\int\sqrt{1+u_n^2}\,d\mm$ converge to $\int\sqrt{1+u^2}\,d\mm$. We can now use
the strict convexity of $z\mapsto\sqrt{1+z^2}$ (see for instance Exercise~1.20 of \cite{AFP}) to conclude.
\end{proof} 

By the lower semicontinuous nature of both functionals, we believe that equality always holds in
\eqref{eq:goodper12} at least on $f\in H^{1,1}(X,d,\mm)$, i.e.
$$
P(E_f,B\times\R)=\int_B\sqrt{1+|\nabla f|_{*,1}^2}\,d\mm\qquad\text{for all $B\subset X$ Borel.}
$$
In the next remark, instead, we compare with the relaxation point of view of \cite{HKLL14}.

\begin{remark} {\rm Le us discuss the question of the validity of the equality
\begin{equation}\label{eq:june18b}
\inf\biggl\{\liminf_{n\to\infty}\int_X\sqrt{1+|\nabla f_n|^2}\,d\mm:\ f_n\in{\rm Lip}_{\rm loc}(X),\,\,
f_n\to f\,\,\text{in $L^1(X,\mm)$}\biggr\}=\int_X\sqrt{1+|\nabla f|_{*,1}^2}\,d\mm
\end{equation}
for $f\in H^{1,1}(X,d,\mm)$.
Notice that the inequality $\geq$ follows at once from the lower semicontinuity of the perimeter and 
\eqref{eq:june18a}, \eqref{eq:goodper12} with $B=X$. We can prove that equality holds if and only
if $|\nabla f|_{*,1}=|\nabla f|_{w,1}$ $\mm$-a.e. in $X$. Indeed, if equality holds, a diagonal argument provides
$f_n\in {\rm Lip}_{\rm loc}(X)$ such that $f_n\to f$ in $L^1(X,\mm)$ and
$$
\limsup_{n\to\infty}\int_X\sqrt{1+|\nabla f_n|^2}\,d\mm\leq\int_X\sqrt{1+|\nabla f|_{*,1}^2}\,d\mm.
$$
If we denote by $\mu$ a weak limit of $|\nabla f_n|$, we obviously have $\mu\geq |Df|$ on
open sets $A$, hence $\mu\geq |\nabla f|_{*,1}\mm$. From Lemma~\ref{lem:strictconvexity} we obtain
that $\mu\ll\mm$ and that $|\nabla f_n|\to |\nabla f|_{*,1}$ in $L^1(X,\mm)$. As in \cite{Sh}, this implies
the existence of representatives $\tilde f$ of $f$ and $g$ of $|\nabla f|_{*,1}$ such that
$$
|\tilde f(\gamma(1))-\tilde f(\gamma(0))|\leq\int_\gamma g
$$
for ${\rm Mod}_1$-a.e. curve $\gamma$, hence $|\nabla f|_{*,1}\geq |\nabla f|_{w,1}$ $\mm$-a.e. in $X$.
}
\end{remark}

\section{Appendix: proof of the inequality $|\nabla f|_*\geq |\nabla f|_{*,c}$}\label{sec:appendix}

In \cite{AGS11a}, a very detailed
analysis of the fine properties of the Hopf-Lax semigroup
\begin{equation}\label{eq:Prato6}
Q_tg(w):=\inf_{w'\in W} g(w')+\frac{1}{2t}d_W^2(w',w)
\end{equation}
in a metric space $(W,d_W)$ has been made. The analysis is based
on the quantities
$$
D^+_g(w,t):=\sup\limsup_{n\to\infty} d_W(w,w_n'),\qquad
D^-_g(w,t):=\inf\liminf_{n\to\infty} d_W(w,w_n),
$$
where the supremum and the infimum run among all minimizing
sequences $(w_n)$ in \eqref{eq:Prato6}. These quantities reduce
respectively to the maximum and minimum distance from $w$ of
minimizers in the locally compact case. Confining for simplicity our
discussion to the case of bounded functions $g:W\to\R$, which suffices for our
purposes, for all $t>0$ it has been shown that (see Propositions~3.2, 3.3 and 3.4 of \cite{AGS11a}):
\begin{itemize}
\item[(a)] $D^+_g$ and $D^-_g$ are
respectively upper and lower semicontinuous in $W\times (0,\infty)$ and,
given $w$, $D^+_g(w,t)=D^-_g(w,t)$ with at most countably many exceptions;
\item[(b)] $Q_t g$ is Lipschitz in $X$, with ${\rm Lip}(Q_t g)\leq 2\sqrt{{\rm osc}(g)/t}$, where ${\rm osc}(f)=\sup f-\inf f$;
\item[(c)] $D^-_g(\cdot,t)/t$ is an upper gradient of $Q_tg$;
\item[(d)] the following pointwise equality holds:
\begin{equation}\label{eq:Prato4}
\frac{d^+}{d t}Q_tg(w)+\frac{(D^+_g(w,t))^2}{2t^2}=0,
\end{equation}
where $d^+/d t$ stands for right derivative (part
of the statement is its existence at every point).
\end{itemize}
Notice that, since $D^+_g(\cdot,t)/t\geq D^-_g(\cdot,t)/t$ is an upper
semicontinuous upper gradient of $Q_t g$, it bounds the slope of $Q_t g$ from above.
Therefore \eqref{eq:Prato4} implies the Hamilton-Jacobi subsolution property $\tfrac{d^+}{d
t}Q_tg+|\nabla Q_tg|^2/2\leq 0$, but the
sharper form \eqref{eq:Prato4} is often essential for the proofs.

More precisely, our proof of the inequality $|\nabla f|_*\geq |\nabla f|_{*,c}$ is based on a refinement of Lemma~6.16, Lemma~6.17 and 
Proposition~6.18 of \cite{AGS11b} and it is split in these steps:

{\bf Step 1.} {\it Let $g:Z\to\R$ be a bounded function and set
\begin{equation}\label{eq:treponzio}
Q_t g(x,y):=\inf\left\{g(x',y')+\frac 1{2t}\bigl(d^2_X(x,x')+d^2_Y(y,y')\bigr)\right\}.
\end{equation}
In this step we show that for $\Leb{1}$-a.e. $t>0$ one has
\begin{equation}\label{eq:prato8}
\frac{d^+}{dt}Q_t g+\frac 12 |\nabla Q_t g|_c^2\leq 0\qquad\text{$\mm$-a.e. in $Z$.}
\end{equation}}
In (6.40) of \cite{AGS11b}, starting from \eqref{eq:Prato4}, an inequality similar to \eqref{eq:prato8} for
\emph{all $t>0$} has been proved, where instead of $|\nabla Q_t g|_c^2$ the sum $|(Q_t g)^x|_{*,Y}^2+|(Q_t g)^y|_{*,X}^2$
was used. Our refinement \eqref{eq:prato8} still relies on \eqref{eq:Prato4} and the pointwise inequality (see (6.43) of \cite{AGS11b})
\begin{equation}\label{eq:june1}
[D^+_g((x,y),t)]^2\geq
[D^-_{L_{t,y}}(x,t)]^2+[D^-_{R_{t,x}}(y,t)]^2
\end{equation}
where $L_{t,y}(x'):=Q_t^Yg(x',\cdot)(y)$ and
$R_{t,x}(y'):=Q_t^Xg(\cdot,y')(x)$.  By minimizing first w.r.t. to one variable and then with respect to the other variable
in \eqref{eq:treponzio} we have the
easy identities
\begin{equation}\label{eq:identities}
(Q_tg)^y(x)=Q_t^X(L_{t,y})(x),\qquad (Q_tg)^x(y)=Q_t^Y(R_{t,x})(y).
\end{equation}

By (a) we can apply Fubini's theorem to convert $D^-$ into $D^+$ in \eqref{eq:june1}, thus obtaining 
$$
[D^+_g((x,y),t)]^2\geq
[D^+_{L_{t,y}}(x,t)]^2+[D^+_{R_{t,x}}(y,t)]^2\qquad\text{for $\mm$-a.e. $(x,y)\in Z$}
$$
for $\Leb{1}$-a.e. $t>0$. Since $D^+$ is an upper semicontinuous upper gradient and upper semicontinuous
upper gradients bound, as we already said, the slope from above, we can use \eqref{eq:identities} to get eventually \eqref{eq:prato8}.

{\bf Step 2.} (the so-called Kuwada lemma). {\it In this step we show that if $h\in L^\infty(Z,\mm)$ is a probability density, and $h_t$ is the solution to the gradient
flow of $\C_c$ starting from $h$, then $\mu_t=h_t\mm\in\Probabilities{Z}$, $t\mapsto\mu_t$ is locally absolutely continuous from $[0,\infty)$ to
$\Probabilities{Z}$ endowed with the quadratic Wasserstein distance $W_2$. In addition, its metric speed satisfies
\begin{equation}\label{eq:Levico11}
|\dot\mu_t|^2\leq\int_{\{h_t>0\}}\frac{|\nabla h_t|_{*,c}^2}{h_t}\,d\mm\qquad\text{for $\Leb{1}$-a.e. $t>0$}.
\end{equation}}
The proof of this fact uses the Hamilton-Jacobi subsolution property \eqref{eq:prato8} of Step 1 (used in integral form, so its
validity for $\Leb{1}$-a.e. $t>0$ is sufficient), as well as the integration by parts formula
\eqref{eq:bypartsG} with the Laplacian $\Delta_G$ corresponding to the energy $\C_G$ induced by the
pseudo gradient $G(f)=|\nabla f|_c$.

{\bf Step 3.} {\it In this step we conclude the proof of the inequality $|\nabla f|_*\geq |\nabla f|_{*,c}$}. 
Now for any bounded nonnegative $f\in D(\C)$ with $\int f^2\,d\mm=1$ we can find, arguing exactly as in Proposition~6.18 of
\cite{AGS11b}, a sequence of locally Lipschitz
functions $f_n$ convergent to $f$ in $L^2(X,\mm)$ satisfying
\begin{equation}\label{eq:easter}
\limsup_{n\to\infty}\int_Z |\nabla f_n|_{*,c}^2\,d\mm\leq \C(f)
\end{equation}
(the only difference is that the gradient flow of $\C_c$ has to be used, as in the previous steps). Therefore $f\in D(\C_c)$
and $\C_c(f)\leq\C(f)$ which, in combination with \eqref{eq:cdc1}, gives the result. 
Using invariance under addition by constants and homogeneity, we extend the result to all bounded functions $f$.
Eventually a truncation argument and the locality of weak gradients provide the result for general $f$. The preliminary
reduction to nonnegative and normalized $f$'s is necessary in view of Step 2, because we use in the construction of
$f_n$ the estimate \eqref{eq:Levico11} on metric derivative for a $\Probabilities{Z}$-valued map. 

We briefly sketch, for the reader's convenience the argument leading to \eqref{eq:easter}, referring to \cite[Theorem~6.2]{AGS11a} and
\cite[Proposition~6.18]{AGS11b} for more details. By homogeneity and invariance under addition of constants we can assume, besides
$\int f^2\,d\mm=1$, that
$c^{-1}\ge f\ge c>0$ $\mm$-almost everywhere in $Z$. We consider the gradient flow $(h_t)$ of $\C_c$ with
initial datum $h:=f^2$, setting $\mu_t=h_t\mm$. The maximum principle yields $c^{-2}\geq
h_t\geq c^2$ and a standard argument based on chain rule and integration by parts \eqref{eq:bypartsG}
yields the energy dissipation identity
\begin{equation}\label{eq:Levico12}
\frac{d}{d t}\int h_t\log h_t\,d\mm=-\int_{\{h_t>0\}}\frac{|\nabla h_t|_{*,c}^2}{h_t}\,d\mm
\qquad\text{for $\Leb{1}$-a.e. $t>0$.}
\end{equation}

Let $g=h^{-1}|\nabla h|_*$;  by the chain rule we know
that $\log h$ is Sobolev along almost every curve and we can use the same
argument of \cite[Theorem~6.2]{AGS11a} to get
$$
\int\big(h\log h-h_t \log h_t\big)\,d\mm\leq\int \log
h(h-h_t)\,d\mm\leq \Big(\int_0^t\int g^2 h_s\,d\mm\,d s\Big)^{1/2}\Big(
\int_0^t|\dot \mu_s|^2\,d s\Big)^{1/2}.
$$
Now, inequality \eqref{eq:Levico11} gives
\begin{eqnarray*}
  \int\big(h\log h-h_t \log h_t\big)\,d\mm&\le&
  \frac 12 \int_0^t\int  g^2 h_s\,d\mm\,d s+
  \frac 12 \int_0^t |\dot\mu_s|^2\,d s
  \\&\le&
  \frac 12 \int_0^t\int  g^2 h_s\,d\mm\,d s+
  \frac 12 \int_0^t \int_{\{h_s>0\}}\frac{|\nabla h_s|_{*,c}^2}{h_s}\,d\mm\,d s.
\end{eqnarray*}
Recalling the entropy dissipation formula \eqref{eq:Levico12} we
obtain
\[
  \int_0^t \int_{\{h_s>0\}}\frac {|\nabla h_s|_{*,c}^2}{h_s}\,d\mm\,d s\le\int_0^t\int g^2h_s\,d\mm\,d s.
\]
Now, the chain rule and the identity $g=2f^{-1}|\nabla f|_*$ give
$\int_0^t\C_c(\sqrt{h_s})\,d s\leq\int_0^t\int|\nabla f|_*^2f^{-2}
h_s\,d\mm\,d s$, so that dividing by $t$ and passing to the limit
as $t\downarrow0$ we get 
$$
\limsup_{t\downarrow 0}\frac 1 t
\int_0^t\C_c(\sqrt{h_s})\,d s\leq\int_X|\nabla f|_*^2\,d\mm.
$$
Therefore there exists $s_i\downarrow 0$ with
$\limsup_i\C_c(\sqrt{h_{s_i}})\leq\C(f)$.
Since $\sqrt{h_{s_i}}$ are
equibounded and converge strongly to $f$ in $L^2(Z,\mm)$ as
$s\downarrow0$, a diagonal argument provides \eqref{eq:easter}.


\begin{thebibliography}{99}

\bibitem[AFP]{AFP} {\sc L.~Ambrosio, N.~Fusco, D.~Pallara,} {\em Functions of bounded variation and free discontinuity
problems.} Oxford University Press, 2000. 

\bibitem[AGS05]{AGS05}
{\sc L.~Ambrosio, N.~Gigli, and G.~Savar{\'e}}, {\em Gradient flows in
  metric spaces and in the space of probability measures}. Lectures in
  Mathematics ETH Z\"urich, Birkh\"auser Verlag, Basel, second~ed., 2008.

\bibitem[AGS11a]{AGS11a}
{\sc L.~Ambrosio, N.~Gigli and G.~Savar{\'e}}, \emph{Calculus and heat
  flow in metric measure spaces and applications to spaces with {R}icci bounds
  from below}. Inventiones Mathematicae, {\bf 195} (2014), 289--391.

\bibitem[AGS11b]{AGS11b}
{\sc L.~Ambrosio, N.~Gigli and G.~Savar{\'e}}, \emph{Metric measure spaces with {R}iemannian {R}icci curvature
  bounded from below}. Duke Math. J., {\bf 163} (2014), 1405--1490.
  
\bibitem[AGS12]{AGS12}
{\sc L.~Ambrosio, N.~Gigli and G.~Savar{\'e}}, \emph{Density of Lipschitz functions and 
equivalence of weak gradients in metric measure spaces}. Rev. Mat. Iberoam., {\bf 29} (2013), 
969--996.

\bibitem[AmDi]{AmDi}
{\sc L.~Ambrosio and S.~Di Marino}, \emph{Equivalent definitions of $BV$ space and of total variation on metric measure spaces.}
Journal of Functional Analysis, {\bf 266} (2014), 4150--4188.

\bibitem[AmCoDi]{AmCoDi}
{\sc L.~Ambrosio, M.~Colombo and S.~Di Marino}, \emph{Sobolev spaces in metric measure spaces: reflexivity and lower semicontinuity of slope.}
ArXiv 1212.3779 (2012), to appear on Advanced Studies in Pure Mathematics.

\bibitem[AmDiSa]{AmDiSa}
{\sc L.~Ambrosio, S.~Di Marino and G.~Savar\'e}, \emph{On the duality between $p$-modulus and probability
measures}. ArXiv 1311.1381 (2013).

\bibitem[Bre]{Bre}
{\sc H.~Breziz}, \emph{Analyse Fonctionelle.} Masson, Paris, 1983.

\bibitem[Chee]{Chee}
{\sc J.~Cheeger}, {\em Differentiability of {L}ipschitz functions on metric
  measure spaces}. Geom. Funct. Anal., {\bf 9} (1999), 428--517.

\bibitem[DiSp]{DiSp}
{\sc S.~Di Marino and G.~Speight}, \emph{The $p$ weak gradient depends on $p$.} ArXiv 1311.4171, 2013.

\bibitem[Gi12]{Gigli12}
{\sc N.~Gigli,} \emph{On the differential structure of metric measure spaces and
  applications},  ArXiv 1205.6622 (2012), to appear on Memoirs AMS.
  
\bibitem[GiHa]{GiHa} {\sc N.~Gigli, B.-X.~Han}, \emph{Unifications of weak gradients on metric measure spaces.}
Paper in progress.

\bibitem[GoTr]{GoTr}
{\sc V.~Gol'dshtein and M.~Troyanov,} \emph{Axiomatic theory of Sobolev spaces.} Expo Math., {\bf 19} (2001),  289--336.

\bibitem[Hei]{Hei}
{\sc J.~Heinonen,} \emph{Nonsmooth calculus.} Bull. Amer. Mat. Soc., {\bf 44} (2007), 163--232.

\bibitem[HK]{HK} 
{\sc P.~Hajlasz and P.~Koskela,} \emph{Sobolev Met Poincar\'{e}}, Memoirs AMS, {\bf 688} (2000). 

\bibitem[HKL]{HKL13}
{\sc H.~Hakkarainen, J.~Kinnunen, P.~Lahti},
\emph{Regularity of minimizers of the area functional in metric spaces.} 
Preprint, 2013, to appear on Adv. Calc. Var. 

\bibitem[HKLL14]{HKLL14}
{\sc H.~Hakkarainen, J.~Kinnunen, P.~Lahti, and P.~Lehtel\"a}, 
\emph{Relaxation and integral representation for functionals of linear growth on metric measure spaces.}
ArXiv 1401.5717 (2014).

\bibitem[KLS13]{KLS13}
{\sc R.~Korte, P.~Lahti, and N.~Shanmugalingam}, \emph{Semmes family of curves and a characterization of
functions of bounded variation in terms of curves.} Preprint, 2013.

\bibitem[Mi]{Mi}
{\sc M.~Miranda Jr}, \emph{Functions of bounded variation on ``good'' metric spaces}. J. Math. Pures Appl., 
{\bf 82} (2003), 975--1004.

\bibitem[Sh]{Sh}
{\sc N.~Shanmugalingham,} \emph{Newtonian spaces: an extension of Sobolev spaces to metric measure spaces.}
Rev. Mat. Iberoam., {\bf 16} (2000), 243--279.

\bibitem[HKST]{HKST}
{\sc J.~Heinonen, P.~Koskela, N.~Shanmugalingham and J.~Tyson}, \emph{Sobolev spaces on metric measure spaces: an
approach based on upper gradients.} New Mathematical Monographs,
Cambridge University Press, to appear.

\end{thebibliography}
\end{document}